\documentclass[english]{article}
\usepackage[T1]{fontenc}
\usepackage[latin9]{inputenc}
\usepackage{geometry}
\usepackage{float}
\usepackage{amsmath}
\usepackage{amsthm}
\usepackage{amssymb}

\makeatletter

\floatstyle{ruled}
\newfloat{algorithm}{tbp}{loa}
\providecommand{\algorithmname}{Algorithm}
\floatname{algorithm}{\protect\algorithmname}

\theoremstyle{plain}
\newtheorem{thm}{\protect\theoremname}[section]
\theoremstyle{plain}
\newtheorem{lem}[thm]{\protect\lemmaname}
\theoremstyle{remark}
\newtheorem{rem}[thm]{\protect\remarkname}
\ifx\proof\undefined
\newenvironment{proof}[1][\protect\proofname]{\par
	\normalfont\topsep6\p@\@plus6\p@\relax
	\trivlist
	\itemindent\parindent
	\item[\hskip\labelsep\scshape #1]\ignorespaces
}{%
	\endtrivlist\@endpefalse
}
\providecommand{\proofname}{Proof}
\fi
\theoremstyle{plain}
\newtheorem{prop}[thm]{\protect\propositionname}
\theoremstyle{plain}
\newtheorem{fact}[thm]{\protect\factname}

\@ifundefined{date}{}{\date{}}
\makeatother

\usepackage{babel}
\providecommand{\factname}{Fact}
\providecommand{\lemmaname}{Lemma}
\providecommand{\propositionname}{Proposition}
\providecommand{\remarkname}{Remark}
\providecommand{\theoremname}{Theorem}

\begin{document}
\global\long\def\E{\mathbb{\mathbb{E}}}%
\global\long\def\F{\mathcal{F}}%
\global\long\def\R{\mathbb{R}}%
\global\long\def\tn{\widetilde{\nabla}f}%
\global\long\def\hn{\widehat{\nabla}f}%
\global\long\def\n{\nabla f}%
\global\long\def\indicator{\mathbf{1}}%
\global\long\def\mf{f(x^{*})}%
\global\long\def\breg{\mathbf{D}_{\psi}}%
\global\long\def\dom{\mathcal{X}}%
\global\long\def\norm#1{\left\lVert #1\right\rVert }%

\title{Improved Convergence in High Probability of Clipped Gradient Methods
with Heavy Tails}
\author{Ta Duy Nguyen\thanks{Department of Computer Science, Boston University, \texttt{taduy@bu.edu}.} \and
Alina Ene\thanks{Department of Computer Science, Boston University, \texttt{aene@bu.edu}.} \and
Huy L. Nguyen\thanks{Khoury College of Computer and Information Science, Northeastern University, \texttt{hu.nguyen@northeastern.edu}.}}
\maketitle
\begin{abstract}
In this work, we study the convergence \emph{in high probability}
of clipped gradient methods when the noise distribution has heavy
tails, ie., with bounded $p$th moments, for some $1<p\le2$. Prior
works in this setting follow the same recipe of using concentration
inequalities and an inductive argument with union bound to bound the
iterates across all iterations. This method results in an increase
in the failure probability by a factor of $T$, where $T$ is the
number of iterations. We instead propose a new analysis approach based
on bounding the moment generating function of a well chosen supermartingale
sequence. We improve the dependency on $T$ in the convergence guarantee
for a wide range of algorithms with clipped gradients, including stochastic
(accelerated) mirror descent for convex objectives and stochastic
gradient descent for nonconvex objectives. This approach naturally
allows the algorithms to use time-varying step sizes and clipping
parameters when the time horizon is unknown, which appears impossible
in prior works. We show that in the case of clipped stochastic mirror
descent, problem constants, including the initial distance to the
optimum, are not required when setting step sizes and clipping parameters.
\end{abstract}

\section{Introduction}

Stochastic optimization is a well-studied area with many applications
in a variety of domains from machine learning, to operation research,
numerical linear algebra and beyond. In contrast with deterministic
algorithms, stochastic algorithms might fail and a pertinent question
is to understand how often this happens and how to make changes to
increase the success rate. This question is especially important in
critical applications where failure is not acceptable but it is also
crucial for many machine learning applications where each run is very
expensive and time consuming. Fortunately, the standard stochastic
gradient descent (SGD) algorithm has been shown to converge with high
probability under a light-tailed noise distribution such as the Gaussian
distribution \cite{nemirovski2009robust,kakade2008generalization,rakhlin2011making,lan2012optimal,hazan2014beyond,harvey2019tight}.
However, recent observations in deep learning applications reveal
a more challenging optimization landscape, thus requiring new changes
to the algorithm \cite{simsekli2019tail,csimcsekli2019heavy,zhang2020adaptive,gurbuzbalaban2021heavy}.

In this work, we consider the general model for heavy tailed noise
by \cite{zhang2020adaptive}. In this model, the gradient noise is
unbiased and it has bounded $p$th moments, for some $p\in(1,2]$.
While SGD might fail to converge in this setting, \cite{zhang2020adaptive}
showed that SGD with appropriate clipping (clipped-SGD) converges
in expectation. Recent follow-up works \cite{nguyen2023high,liu2023breaking,sadiev2023high}
showed that in fact, the algorithm converges with high probability.
This is a pleasing result, extending the earlier work by \cite{gorbunov2020stochastic}
for $p=2$. However, there are several shortcomings of these results
compared with the corresponding bound in the light-tailed setting.
First, the clipped algorithm uses a fixed step size and a fixed clipping
parameter depending on the number of iterations, thus precluding results
with unknown time horizons. Secondly, the convergence guarantees are
worse than the light-tailed bounds by a $\log T$ factor. These issues
beg a qualitative question:
\begin{center}
\emph{Is heavy-tailed noise inherently harder to deal with than light-tailed
noise?}
\par\end{center}

In this work, we provide answers for the above question with an improved
and general analysis framework for clipped-SGD. Our analysis allows
for time-varying stepsizes and clipping parameters as well as tighter
bounds matching those in the case of light-tailed noise. The analysis
framework is also applicable to different settings, from finding minimizers
of convex functions with arbitrarily large domains using mirror descent,
to finding first-order stationary points for non-convex functions
using gradient descent.

\subsection{Contribution}

The contribution of our work is as follows.

First, we demonstrate a new approach to analyze the convergence in
high probability of clipped gradient methods. The approach is general
and applicable for a wide range of algorithms, including clipped stochastic
mirror descent, clipped stochastic accelerated mirror descent in the
convex setting and clipped stochastic gradient descent in the nonconvex
setting. Our approach is based on the idea of ``whitebox'' concentration
inequalities put forward by \cite{liu2023high}. While this prior
work can only work for light tailed noises, we show new elements that
allow us to analyze heavy tailed noises, i.e, noises with bounded
$p$th moments. 

Second, using this new analysis framework, we give time-optimal convergence
guarantees for both convex and nonconvex objectives. In the convex
setting, our convergence rate is $O\left(T^{\frac{1-p}{p}}\right)$
for arbitrary domains and Bregman divergences to measure distances.
In the nonconvex setting, the convergence rate is $O\left(T^{\frac{2-2p}{3p-2}}\right)$.
Both of these match the lower bounds proven in \cite{raginsky2009information,vural2022mirror,zhang2020adaptive}.

Finally, we show that existing convergence rates for clipped gradient
methods can be tightened up and significantly improved. Specifically,
the $\log T$ factor loss due to the sub-optimality of the previous
analysis is improved. Our framework readily extends to the case where
the time horizon is unknown and allows for time-varying step sizes
and clipping parameters, which appears impossible in prior works.
Furthermore, with very mild assumptions, we are able to give a choice
of clipping parameters and step sizes (for stochastic mirror descent)
that do not depend on the problem parameters, including an upper bound
for the noise variance and the initial distance between the initial
solution and the optimal solution, which are generally unknown in
practice.

\subsection{Related Work}

\paragraph{High probability convergence for light-tailed noises}

Convergence in high probability of stochastic gradient algorithms
has been established for sub-gaussian noises in a number of prior
works, including \cite{nemirovski2009robust,kakade2008generalization,rakhlin2011making,lan2012optimal,hazan2014beyond,harvey2019tight}
for convex problems with bounded domain (or bounded Bregman diameter)
or with strong convexity. Other works by \cite{li2020high,madden2020high,li2022high}
study convergence of several variants of SGD for nonconvex objectives.
These works consider lighter tailed distributions than the ones with
bounded variance, ranging from sub-gaussian to sub-Weibull noises.
Light tailed noises do not require gradient clipping, and are generally
easier to analyze than heavy tailed counterparts, for which the bias
of clipping the gradients propagating through the algorithm requires
careful attention. The most relevant to ours in this line of work
is the one by \cite{liu2023high}. This work develops a whitebox approach
to analyzing stochastic (accelerated) mirror descent for convex objectives
and stochastic gradient descent for nonconvex functions. In the convex
setting, by incorporating both the function value gap and the distance
between the iterates and the optimum in a properly defined martingale,
this work can leverage the mechanism of bounding the moment generating
function to achieve a tight bound for the convergence rate. Our work
is based on a similar idea but departs significantly to cope with
the heavier tailed noise. In particular, we present a novel choice
of coefficients for the martingale sequence that are dependent on
the history as opposed to the ones that are independent of the past
in \cite{liu2023high}. This idea is general and applicable in both
stochastic (accelerated) mirror descent (for convex functions) and
stochastic gradient descent (for nonconvex functions). 

\paragraph{High probability convergence for noises with bounded variance and
heavy tails}

The design of new gradient algorithms and their analysis in the presence
of heavy tailed noises has drawn significant recent interest. Starting
from the work of \cite{pascanu2012understanding} which proposes Clipped-SGD,
the recent works of \cite{simsekli2019tail,csimcsekli2019heavy,zhang2020adaptive,gurbuzbalaban2021heavy}
give new motivation from the observation of how the gradient noise
behaves in practice, wherein only Clipped-SGD can be useful. The notion
of heavy tailed noise is defined as noises with bounded $p$th moments,
for some $1<p\le2$. In particular, across various deep learning tasks,
the gradients exhibit infinite variance ($p<2$). This is problematic,
as shown by \cite{zhang2020adaptive}, because vanilla SGD may not
converge in this case.

While the convergence in expectation of vanilla SGD has been extensively
studied, for example, \cite{ghadimi2013stochastic,nemirovski2009robust,khaled2020better},
only recently has the convergence of Clipped-SGD with heavy tailed
noises been closely examined. \cite{zhang2020adaptive} show the convergence
in expectation of Clipped-SGD for nonconvex functions and prove the
lowerbound for the convergence rate. \cite{nazin2019algorithms,parletta2022high,gorbunov2020stochastic}
use different clipping methods for noises with bounded variance and
show the convergence in high probability for convex and smooth problems
with bounded domain (the first and second) or unconstrained (the third).
Later, \cite{gorbunov2021near} extends the analysis to the non-smooth
function setting using a similar method. \cite{cutkosky2021high}
propose a different variant of Clipped-SGD that incorporates momentum
and show its convergence in high probability for noises with bounded
$p$th moments. In this work, the momentum plays a central role to
guarantee the optimal convergence of the algorithm. The analysis also
requires quite restrictively that the stochastic gradients are bounded.
The very recent works by \cite{nguyen2023high,liu2023breaking,sadiev2023high,zhangparameter}
also focus on this class of noise, giving the (nearly-)optimal convergence
rate for several Clipped-SGD variants. 

Our paper follows this same line of work that studies the convergence
in high probability with heavy tailed noises, and advances the state
of the art in several key aspects. Compared with \cite{Cutkosky19},
our analysis does not require the use of momentum; we show that a
simple clipping strategy with a proper choice of clipping parameters
and step sizes is sufficient for the algorithm to converge optimally.
The bounded gradient assumption is also not necessary. \cite{zhangparameter}
focus on the class of parameter free algorithms aiming at improving
the dependency on the initial distance, which is significantly different
and more complicated. Compared with the works by \cite{gorbunov2020stochastic}
and \cite{nguyen2023high,liu2023breaking,sadiev2023high}, our analysis
is significantly tighter. These works all follow the same recipe of
using concentration inequalities, specifically Freedman-type inequalities
\cite{freedman1975tail,dzhaparidze2001bernstein} as a blackbox. The
key technique is to bound the iterates inductively for all iterations.
This process incurs extra $\log T$ terms in the final convergence
rate; in other words, the success probability goes from $1-\delta$
to $1-T\delta$. In contrast, we apply a white-box approach to the
problem, similarly to the work \cite{liu2023high} for sub-Gaussian
noise. This allows us to bound iterates across all iterations at once,
preserving the same success probability. Notably, this approach also
allows us to deal with the case where the time horizon is unknown,
or when the problem parameters such as the noise upper bound $\sigma$,
the failure probability $\delta$ and perhaps more importantly the
initial distance to the optimum are unknown, which appears impossible
in these prior works. Finally, our work generalize to stochastic (accelerated)
mirror descent, with arbitrary norms and domains.

In the related line of work, \cite{wang2021convergence} show that
vanilla SGD can converge with heavy tailed noise under some special
assumptions, while \cite{vural2022mirror} show convergence in expectation
of stochastic mirror descent for a special choice of mirror maps,
for strongly convex objectives in bounded domains. Lower bounds for
the optimal convergence rate are shown by \cite{raginsky2009information,vural2022mirror}
(for convex settings) and \cite{zhang2020adaptive} (for nonconvex
settings). In both cases, our approach is able to produce optimal
convergence guarantees.

\section{Preliminaries}

We study the problem $\min_{x\in\dom}f(x)$ where $f:\R^{d}\to\R$
and $\dom$ is the domain of the problem. In the convex setting, we
assume that $\dom$ is a convex set but not necessarily compact. We
let $\left\Vert \cdot\right\Vert $ be an arbitrary norm and $\left\Vert \cdot\right\Vert _{*}$
be its dual norm. In the nonconvex setting, we take $\dom$ to be
$\R^{d}$ and consider only the $\ell_{2}$ norm. 

\subsection{Assumptions}

We use the following assumptions:

\textbf{(1) Existence of a minimizer}: In the convex setting, we assume
that $x^{*}\in\arg\min_{x\in\dom}f(x)$. We let $f^{*}=f(x^{*})$.

\textbf{(1') Existence of a finite lower bound}: In the nonconvex
setting, we assume that $f$ admits a finite lower bound, ie., $f^{*}:=\inf_{x\in\R^{d}}f(x)>-\infty$.

\textbf{(2) Unbiased estimator}: We assume that our algorithm is allowed
to query $f$ via a stochastic first-order oracle that returns a history-independent,
unbiased gradient estimator $\hn(x)$ of $\nabla f(x)$ for any $x\in\dom$.
 That is, conditioned on the history and the queried point $x$,
we have $\E\left[\hn(x)\mid x\right]=\n(x)$.

\textbf{(3) Bounded $p$th moment noise}: We assume that there exists
$\sigma>0$ such that for some $1<p\le2$ and for any $x\in\dom$,
$\hn(x)$ satisfies $\E\left[\left\Vert \hn(x)-\n(x)\right\Vert _{*}^{p}\mid x\right]\le\sigma^{p}.$

\textbf{(4) $L$-smoothness}: We consider the class of $L$-smooth
functions: for all $x,y\in\R^{d}$, $\left\Vert \nabla f(x)-\nabla f(y)\right\Vert _{*}\le L\left\Vert x-y\right\Vert .$

\subsection{Gradient clipping operator and notations}

We introduce the gradient clipping operator and its general properties
to be used in Clipped Stochastic Mirror Descent (Algorithm \ref{alg:clipped-smd})
and Clipped Stochastic Gradient Descent (Algorithm \ref{alg:clipped-sgd}).
Let $x_{t}$ be the solution at iteration $t$ of the algorithm of
interest. We denote by $\hn(x_{t})$ the stochastic gradient obtained
by querying the gradient oracle. The clipped gradient estimate $\tn(x_{t})$
is taken as 
\begin{equation}
\tn(x_{t})=\min\left\{ 1,\frac{\lambda_{t}}{\left\Vert \hn(x_{t})\right\Vert _{*}}\right\} \hn(x_{t}).\label{eq:clipping-formula}
\end{equation}
where $\lambda_{t}$ is the clipping parameter used in iteration $t$.
In subsequent sections, we let $\Delta_{t}:=f(x_{t})-f^{*}$ denote
the function value gap at $x_{t}$. We let $\F_{t}=\sigma\left(\hn(x_{1}),\dots,\hn(x_{t})\right)$
be the natural filtration and define the following notations:

\[
\theta_{t}=\tn(x_{t})-\nabla f(x_{t});\quad\theta_{t}^{u}=\tn(x_{t})-\E\left[\tn(x_{t})\mid\F_{t-1}\right];\quad\theta_{t}^{b}=\E\left[\tn(x_{t})\mid\F_{t-1}\right]-\nabla f(x_{t}).
\]
Note that $\theta_{t}^{u}+\theta_{t}^{b}=\theta_{t}.$ Regardless
of the convexity of the function $f$, the following lemma provides
upper bounds for these quantities. These bounds can be found in prior
works \cite{gorbunov2020stochastic,zhang2020adaptive,nguyen2023high,liu2023breaking,sadiev2023high}
for the special case of $\ell_{2}$ norm. The extension to the general
norm follows in the same manner, which we omit in this work. 
\begin{lem}
\label{lem:bias-bounds}We have
\begin{align}
\left\Vert \theta_{t}^{u}\right\Vert _{*}=\left\Vert \tn(x_{t})-\E\left[\tn(x_{t})\mid\F_{t-1}\right]\right\Vert _{*} & \le2\lambda_{t}\label{eq:1-p}
\end{align}
Furthermore, if $\left\Vert \nabla f(x_{t})\right\Vert _{*}\le\frac{\lambda_{t}}{2}$
then
\begin{align}
\left\Vert \theta_{t}^{b}\right\Vert _{*} & =\left\Vert \E\left[\tn(x_{t})\mid\F_{t-1}\right]-\nabla f(x_{t})\right\Vert _{*}\le4\sigma^{p}\lambda_{t}^{1-p};\label{eq:2-p}\\
\E\left[\left\Vert \theta_{t}^{u}\right\Vert _{*}^{2}\right] & =\E\left[\left\Vert \tn(x_{t})-\E_{t}\left[\tn(x_{t})\right]\right\Vert _{*}^{2}\mid\F_{t-1}\right]\le40\sigma^{p}\lambda_{t}^{2-p}.\label{eq:4-p}
\end{align}
\end{lem}

Finally, we state a simple but important lemma that bounds the moment
generating function of a zero-mean bounded random variable. The proof
can be found in, for example, \cite{li2020high}.
\begin{lem}
\label{lem:moment-inequality}Let $X$ be a random variable such that
$\E\left[X\right]=0$ and $\left|X\right|\le R$ almost surely. Then
for $0\le\lambda\le\frac{1}{R}$ 
\begin{align*}
\E\left[\exp\left(\lambda X\right)\right] & \le\exp\left(\frac{3}{4}\lambda^{2}\E\left[X^{2}\right]\right).
\end{align*}
\end{lem}

\section{Clipped Stochastic Mirror Descent\label{sec:Clipped-Stochastic-Mirror}}

\begin{algorithm}
\caption{Clipped-SMD}
\label{alg:clipped-smd}

Parameters: initial point $x_{1}$, step sizes $\left\{ \eta_{t}\right\} $,
clipping parameters $\left\{ \lambda_{t}\right\} $, $\psi$ is $1$-strongly
convex wrt $\left\Vert \cdot\right\Vert $

for $t=1$ to $T$ do

$\quad$$\tn(x_{t})=\min\left\{ 1,\frac{\lambda_{t}}{\left\Vert \hn(x_{t})\right\Vert _{*}}\right\} \hn(x_{t})$

$\quad$$x_{t+1}=\arg\min_{x\in\dom}\left\{ \eta_{t}\left\langle \tn(x_{t}),x\right\rangle +\breg\left(x,x_{t}\right)\right\} $
\end{algorithm}

In this section, we present and analyze the Clipped Stochastic Mirror
Descent algorithm (Algorithm \ref{alg:clipped-smd}). We define the
Bregman divergence $\breg(x,y)=\psi(x)-\psi(y)-\left\langle \nabla\psi(y),x-y\right\rangle $
where $\psi:\R^{d}\to\R$ is a $1$-strongly convex differentiable
function with respect to the norm $\left\Vert \cdot\right\Vert $
on $\dom$. We assume that $\mathrm{dom}\psi=\R^{d}$ or more generally
$\dom\subseteq\mathrm{int}\left(\mathrm{dom}\psi\right)$ for convenience.
Algorithm \ref{alg:clipped-smd} is a generalization of Clipped-SGD
for convex functions to an arbitrary norm. The only difference from
the standard Stochastic Mirror Descent algorithm is the use of the
clipped gradient $\tn(x_{t})$ in place of the true stochastic gradient
$\hn(x_{t})$ when computing the new iterate $x_{t+1}$.

Prior works such as \cite{gorbunov2020stochastic} only consider the
setting where the global minimizer lies in $\dom$. Our algorithm
in this section does not require this restriction and instead only
uses the following mild assumption from \cite{nazin2019algorithms}:

\textbf{(5) Existence of a good gradient estimate}: We assume to have
access to a vector $g_{0}\in\R^{d}$ and a constant $\mu\ge0$ such
that at a point $x_{0}\in\dom$ we have $\left\Vert g_{0}-\n(x_{0})\right\Vert _{*}\le\mu\sigma$.
This assumption appears in the work of \cite{nazin2019algorithms}
that shows high probability convergence of a version of clipped Stochastic
Mirror Descent in bounded domain for noises with bounded variance.
This assumption is justified as follows. For unconstrained problems
and constrained problems for which we know that the optimum $x^{*}$
is also the global minimizer, we can simply take $x_{0}=x^{*}$, $g_{0}=0$
and $\mu=0$. Otherwise, one can choose some $x_{0}\in\dom$ and a
success probability $\epsilon$ and use $O\left(\ln\frac{1}{\epsilon}\right)$
queries to the gradient oracle to get the guarantee (from \cite{minsker2015geometric})
that $\Pr\left[\left\Vert g_{0}-\n(x_{0})\right\Vert _{*}>\mu\sigma\right]\le\epsilon$,
where $g_{0}$ is taken as the geometric mean of the stochastic gradients.

We will first state the final convergence guarantee for this algorithm
in the following theorem.
\begin{thm}
\label{thm:convex-convergence}Assume that $f$ satisfies Assumption
(1), (2), (3), (4) and (5). Let $\gamma=\max\left\{ \log\frac{1}{\delta};1\right\} $;
$R_{0}=\sqrt{2\breg\left(x^{*},x_{0}\right)}$ and $R_{1}=\sqrt{2\breg\left(x^{*},x_{1}\right)}$

1. For known $T$, we choose $\lambda_{t}$ and $\eta_{t}$ such that
\begin{align*}
\lambda_{t} & =\lambda=\max\left\{ \left(\frac{26T}{\gamma}\right)^{1/p}\sigma;2\left(2LR_{1}+LR_{0}+\mu\sigma+\left\Vert g_{0}\right\Vert _{*}\right)\right\} \text{, and }\\
\eta_{t} & =\eta=\frac{R_{1}}{24\lambda_{t}\gamma}=\frac{R_{1}}{24\gamma}\min\left\{ \left(\frac{26T}{\gamma}\right)^{-1/p}\sigma^{-1};\frac{1}{2}\left(2LR_{1}+LR_{0}+\mu\sigma+\left\Vert g_{0}\right\Vert _{*}\right)^{-1}\right\} ,
\end{align*}
Then with probability at least $1-\delta$
\begin{align*}
\frac{1}{T}\sum_{t=2}^{T+1}\Delta_{t} & \le48R_{1}\max\left\{ 26^{\frac{1}{p}}T^{\frac{1-p}{p}}\sigma\gamma^{\frac{p-1}{p}};2\left(2LR_{1}+LR_{0}+\mu\sigma+\left\Vert g_{0}\right\Vert _{*}\right)T^{-1}\gamma\right\} =O\left(T^{\frac{1-p}{p}}\right).
\end{align*}

2. For unknown $T$, we choose 
\begin{align*}
\lambda_{t} & =\max\left\{ \left(\frac{52t\left(1+\log t\right)^{2}}{\gamma}\right)^{1/p}\sigma;2\left(2LR_{1}+LR_{0}+\mu\sigma+\left\Vert g_{0}\right\Vert _{*}\right)\right\} \text{, and }\\
\eta_{t} & =\frac{R_{1}}{24\lambda_{t}\gamma}=\frac{R_{1}}{24\gamma}\min\left\{ \left(\frac{52t\left(1+\log t\right)^{2}}{\gamma}\right)^{-1/p}\sigma^{-1};\frac{1}{2}\left(2LR_{1}+LR_{0}+\mu\sigma+\left\Vert g_{0}\right\Vert _{*}\right)^{-1}\right\} ,
\end{align*}
Then with probability at least $1-\delta$
\begin{align*}
\frac{1}{T}\sum_{t=2}^{T+1}\Delta_{t} & \le48R_{1}\max\left\{ 52^{\frac{1}{p}}T^{\frac{1-p}{p}}\left(1+\log T\right)^{\frac{2}{p}}\sigma\gamma^{\frac{p-1}{p}};2\left(2LR_{1}+LR_{0}+\mu\sigma+\left\Vert g_{0}\right\Vert _{*}\right)T^{-1}\gamma\right\} =\widetilde{O}\left(T^{\frac{1-p}{p}}\right).
\end{align*}
\end{thm}
\begin{rem}
This theorem shows that the convergence rate for the first case is
$O\left(T^{\frac{1-p}{p}}\right)$ and for the second $\widetilde{O}\left(T^{\frac{1-p}{p}}\right)$.
This rate is known to be optimal, as shown in \cite{raginsky2009information,vural2022mirror}.
The above guarantees are also adaptive to $\sigma$, i.e., when $\sigma\to0$,
we obtain the standard $O\left(T^{-1}\right)$ convergence rate of
deterministic mirror descent.
\end{rem}
\begin{rem}
The term $LR_{0}+\mu\sigma+\left\Vert g_{0}\right\Vert _{*}$ in the
above guarantees comes from the inexact estimation $\left\Vert g_{0}\right\Vert _{*}$
of $\left\Vert \n(x_{0})\right\Vert _{*}$.  If we assume that the
global optimum lies in the domain $\dom$, we can simply select $x_{0}=x^{*}$
and this term will disappear. If otherwise we assume to know the exact
value of $\left\Vert \n(x_{0})\right\Vert _{*}$, this term becomes
$LR_{0}+\left\Vert \n(x_{0})\right\Vert _{*}$.
\end{rem}
Before delving into the analysis, let us compare the above theorem
with the convergence guarantees for Clipped-SGD in \cite{gorbunov2020stochastic}
(for $p=2$) and \cite{nguyen2023high,sadiev2023high} (for general
$1<p\le2$). In the first case when the time horizon $T$ is known,
the convergence in Theorem \ref{thm:convex-convergence} does not
have the extra $\log T$ term, compared with the prior works. The
improvement of this term comes from a better concentration analysis
of the martingale difference sequence. Another restriction in these
prior works is that they strongly require that time horizon is known
to set the proper step size and clipping parameters. This means there
is no immediate way to remove this requirement in the analysis. In
contrast, Theorem \ref{thm:convex-convergence} can naturally generalize
for unknown $T$ with the extra $\log T$ term coming from the cost
of not knowing the time horizon. In fact, we can go one step further
and remove the requirement of the constants $\sigma$, $\delta$ and
$R_{1}$ when setting the step size and clipping parameters. We give
the explicit statement in Theorem \ref{thm:convex-remove-constants}.

We will start the analysis by the following basic lemma.
\begin{lem}
\label{lem:convex-basic-inequality}Assume that $f$ satisfies Assumption
(1), (2), (3), (4) and $\eta_{t}\le\frac{1}{4L}$, the iterate sequence
$(x_{t})_{t\ge1}$ output by Algorithm \ref{alg:clipped-smd} satisfies
the following:
\begin{align*}
\eta_{t}\Delta_{t+1} & \le\breg\left(x^{*},x_{t}\right)-\breg\left(x^{*},x_{t+1}\right)+\eta_{t}\left\langle x^{*}-x_{t},\theta_{t}^{u}\right\rangle +\eta_{t}\left\langle x^{*}-x_{t},\theta_{t}^{b}\right\rangle \\
 & \quad+2\eta_{t}^{2}\left(\left\Vert \theta_{t}^{u}\right\Vert _{*}^{2}-\E\left[\left\Vert \theta_{t}^{u}\right\Vert _{*}^{2}\mid\F_{t-1}\right]\right)+2\eta_{t}^{2}\E\left[\left\Vert \theta_{t}^{u}\right\Vert _{*}^{2}\mid\F_{t-1}\right]+2\eta_{t}^{2}\left\Vert \theta_{t}^{b}\right\Vert _{*}^{2}.
\end{align*}
\end{lem}
\begin{rem}
In the appendix, we give a more general statement for the case when
$f$ satisfies 
\begin{align*}
f(y)-f(x) & \le\left\langle \n(x),y-x\right\rangle +G\left\Vert y-x\right\Vert +\frac{L}{2}\left\Vert y-x\right\Vert ^{2},\quad\forall y,x\in\dom.
\end{align*}
This condition is satisfied by both Lipschitz functions (when $L=0$)
and smooth functions (when $G=0$). The proof, which follows from
\cite{lan2020first}, can be found in the appendix. We consider below
only the case of smooth functions ($G=0$), but the analysis can be
naturally extended to the general case, using the general statement
in the appendix.
\end{rem}
In lemma \ref{lem:convex-basic-inequality}, we already decompose
the RHS into appropriate terms that allow us to define a martingale.
The idea of this decomposition can be found in \cite{gorbunov2020stochastic}.
In the same work, the authors analyze two different martingale difference
sequences: $\left(\eta_{t}\left\langle x^{*}-x_{t},\theta_{t}^{u}\right\rangle \right)_{t\ge1}$
and $\left(2\eta_{t}^{2}\left(\left\Vert \theta_{t}^{u}\right\Vert _{*}^{2}-\E\left[\left\Vert \theta_{t}^{u}\right\Vert _{*}^{2}\mid\F_{t-1}\right]\right)\right)_{t\ge1}$
separately in an inductive manner. Here, they bound the distance $\left\Vert x^{*}-x_{t}\right\Vert $
over all iterations by using union bound. This allows the martingale
difference sequences to satisfy the necessary boundedness condition
in order to apply Freedman's inequality. However, due to the union
bound, the success probability goes from $1-\delta$ to $1-T\delta$,
which is suboptimal. We will tighten the analysis by delving into
the mechanism behind concentration inequalities. We start by defining
the following terms for $t\ge1$:
\begin{align*}
Z_{t} & =z_{t}\left(\eta_{t}\Delta_{t+1}+\breg\left(x^{*},x_{t+1}\right)-\breg\left(x^{*},x_{t}\right)-\eta_{t}\left\langle x^{*}-x_{t},\theta_{t}^{b}\right\rangle -2\eta_{t}^{2}\left\Vert \theta_{t}^{b}\right\Vert _{*}^{2}-2\eta_{t}^{2}\E\left[\left\Vert \theta_{t}^{u}\right\Vert _{*}^{2}\mid\F_{t-1}\right]\right)\\
 & \quad-\left(\frac{3}{8\lambda_{t}^{2}}+24z_{t}^{2}\eta_{t}^{4}\lambda_{t}^{2}\right)\E\left[\left\Vert \theta_{t}^{u}\right\Vert ^{2}\mid\F_{t-1}\right]\\
\mbox{where }z_{t} & =\frac{1}{2\eta_{t}\lambda_{t}\max_{i\le t}\sqrt{2\breg\left(x^{*},x_{i}\right)}+16Q\eta_{t}^{2}\lambda_{t}^{2}}
\end{align*}
for a constant $Q\ge1$ and
\begin{align*}
S_{t} & =\sum_{i=1}^{t}Z_{i}
\end{align*}
We introduce the following Lemma \ref{lem:convex-key-inequality},
whose proof will offer the insight into the main technique in this
paper. The technique to prove this lemma is similar to the standard
way of bounding the moment generation function in proving concentration
inequalities, such as Freedman's inequality \cite{freedman1975tail,dzhaparidze2001bernstein}.
The main challenge in this lemma is to find a way to leverage the
structure of Clipped-SMD. In this case, we have to choose the suitable
coefficients $z_{t}$. 

We compare this technique with the one presented in \cite{liu2023high}
for analyzing SMD with sub-gaussian noises. Both are based on the
idea of analyzing the martingale difference sequence in a ``white-box''
manner. In the prior work, thanks to the light tailed noises, the
coefficients $z_{t}$ can be chosen only depending on the problem
parameters, and independently of the algorithm history. This work
utilizes the distance $\breg\left(x^{*},x_{t}\right)$ to absorb the
incurred error during the analysis. In our case, this approach does
not go through. To use Lemma \ref{lem:moment-inequality} to bound
the moment generation function, we have to make sure that $z_{t}\le\frac{1}{R}$
for $R$ being an upper bound for the martingale elements. The key
novel idea here is that we can choose $z_{t}$ depending on the past
iterates. This choice ensures the condition of Lemma \ref{lem:moment-inequality}.

Another difference between the prior work \cite{liu2023high} and
Lemma \ref{lem:convex-key-inequality} is that in the former, the
bound for the moment generation function immediate gives a constant
bound in the RHS of (\ref{eq:convex-key-inequality}). This is not
the case here. However, we have establish a relation for the terms
that holds for all time steps with probability $1-\delta$. This is
an improvement over the success probability $1-T\delta$ using the
induction argument as in \cite{gorbunov2020stochastic,nguyen2023high,sadiev2023high}.
Now we specify the choice of $\eta_{t}$ and $\lambda_{t}$. The following
lemma gives a general condition for the choice of $\eta_{t}$ and
$\lambda_{t}$ that gives the right convergence rate in time $T$.
\begin{lem}
\label{lem:convex-key-inequality}For any $\delta>0$, let $E(\delta)$
be the event that for all $1\le k\le T$
\begin{align}
\sum_{t=1}^{k}z_{t}\eta_{t}\Delta_{t+1}+z_{k}\breg\left(x^{*},x_{k+1}\right) & \le z_{1}\breg\left(x^{*},x_{1}\right)+\log\frac{1}{\delta}+\sum_{t=1}^{k}z_{t}\eta_{t}\left\langle x^{*}-x_{t},\theta_{t}^{b}\right\rangle +2\sum_{t=1}^{k}z_{t}\eta_{t}^{2}\left\Vert \theta_{t}^{b}\right\Vert _{*}^{2}\nonumber \\
 & \quad+\sum_{t=1}^{k}\left(\left(2z_{t}\eta_{t}^{2}+\frac{3}{8\lambda_{t}^{2}}+24z_{t}^{2}\eta_{t}^{4}\lambda_{t}^{2}\right)\E\left[\left\Vert \theta_{t}^{u}\right\Vert _{*}^{2}\mid\F_{t-1}\right]\right)\label{eq:convex-key-inequality}
\end{align}
Then $\Pr\left[E(\delta)\right]\ge1-\delta$.
\end{lem}
\begin{proof}
We have 

\begin{align*}
 & \E\left[\exp\left(Z_{t}\right)\mid\F_{t-1}\right]\times\exp\left(\left(\frac{3}{8\lambda_{t}^{2}}+24z_{t}^{2}\eta_{t}^{4}\lambda_{t}^{2}\right)\E\left[\left\Vert \theta_{t}^{u}\right\Vert _{*}^{2}\mid\F_{t-1}\right]\right)\\
\overset{(a)}{\le} & \E\left[\exp\left(z_{t}\left(\eta_{t}\left\langle x^{*}-x_{t},\theta_{t}^{u}\right\rangle +2\eta_{t}^{2}\left(\left\Vert \theta_{t}^{u}\right\Vert _{*}^{2}-\E\left[\left\Vert \theta_{t}^{u}\right\Vert _{*}^{2}\mid\F_{t-1}\right]\right)\right)\right)\mid\F_{t-1}\right]\\
\overset{(b)}{\le} & \exp\left(\E\left[\frac{3}{4}\left(z_{t}\left(\eta_{t}\left\langle x^{*}-x_{t},\theta_{t}^{u}\right\rangle +2\eta_{t}^{2}\left(\left\Vert \theta_{t}^{u}\right\Vert _{*}^{2}-\E\left[\left\Vert \theta_{t}^{u}\right\Vert _{*}^{2}\mid\F_{t-1}\right]\right)\right)\right)^{2}\mid\F_{t-1}\right]\right)\\
\overset{(c)}{\le} & \exp\left(\left(\frac{3}{2}z_{t}^{2}\eta_{t}^{2}\left\Vert x^{*}-x_{t}\right\Vert ^{2}\E\left[\left\Vert \theta_{t}^{u}\right\Vert _{*}^{2}\mid\F_{t-1}\right]+6z_{t}^{2}\eta_{t}^{4}\E\left[\left\Vert \theta_{t}^{u}\right\Vert _{*}^{4}\mid\F_{t-1}\right]\right)\right)\\
\overset{(d)}{\le} & \exp\left(\left(\frac{3}{2}z_{t}^{2}\eta_{t}^{2}\left\Vert x^{*}-x_{t}\right\Vert ^{2}+24z_{t}^{2}\eta_{t}^{4}\lambda_{t}^{2}\right)\E\left[\left\Vert \theta_{t}^{u}\right\Vert _{*}^{2}\mid\F_{t-1}\right]\right)\\
\overset{(e)}{\le} & \exp\left(\left(\frac{3}{8\lambda_{t}^{2}}+24z_{t}^{2}\eta_{t}^{4}\lambda_{t}^{2}\right)\E\left[\left\Vert \theta_{t}^{u}\right\Vert _{*}^{2}\mid\F_{t-1}\right]\right)
\end{align*}
For $(a)$ we use Lemma \ref{lem:convex-basic-inequality}. For $(b)$
we use Lemma \ref{lem:moment-inequality}. Notice that 
\begin{align*}
\E\left[\left\langle x^{*}-x_{t},\theta_{t}^{u}\right\rangle \right] & =\E\left[\left\Vert \theta_{t}^{u}\right\Vert _{*}^{2}-\E\left[\left\Vert \theta_{t}^{u}\right\Vert _{*}^{2}\mid\F_{t-1}\right]\right]=0
\end{align*}
and since $\left\Vert \theta_{t}^{u}\right\Vert _{*}\le2\lambda_{t}$,
we have
\begin{align*}
 & \left|\eta_{t}\left\langle x^{*}-x_{t},\theta_{t}^{u}\right\rangle +2\eta_{t}^{2}\left(\left\Vert \theta_{t}^{u}\right\Vert _{*}^{2}-\E\left[\left\Vert \theta_{t}^{u}\right\Vert _{*}^{2}\mid\F_{t-1}\right]\right)\right|\\
 & \le\eta_{t}\left\Vert x^{*}-x_{t}\right\Vert \left\Vert \theta_{t}^{u}\right\Vert _{*}+2\eta_{t}^{2}\left(\left\Vert \theta_{t}^{u}\right\Vert _{*}^{2}+\E\left[\left\Vert \theta_{t}^{u}\right\Vert _{*}^{2}\mid\F_{t-1}\right]\right)\\
 & \le2\eta_{t}\lambda_{t}\left\Vert x^{*}-x_{t}\right\Vert +16\eta_{t}^{2}\lambda_{t}^{2}\\
 & \le2\eta_{t}\lambda_{t}\sqrt{2\breg\left(x^{*},x_{t}\right)}+16\eta_{t}^{2}\lambda_{t}^{2}
\end{align*}
Thus $z_{t}\le\frac{1}{2\eta_{t}\lambda_{t}\sqrt{2\breg\left(x^{*},x_{t}\right)}+16\eta_{t}^{2}\lambda_{t}^{2}}$.
For $(c)$ we use the inequalities $(a+b)^{2}\le2a^{2}+2b^{2}$ and
$\E\left[\left(X-\E\left[X\right]\right)^{2}\right]\le\E\left[X^{2}\right]$.
For $(e)$ we use the fact that $\left\Vert \theta_{t}^{u}\right\Vert _{*}\le2\lambda_{t}$
and 
\begin{align*}
z_{t}\eta_{t}\left\Vert x^{*}-x_{t}\right\Vert  & \le\frac{\eta_{t}\left\Vert x^{*}-x_{t}\right\Vert }{2\eta_{t}\lambda_{t}\sqrt{2\breg\left(x^{*},x_{t}\right)}}\le\frac{1}{2\lambda_{t}}.
\end{align*}
We obtain $\E\left[\exp\left(Z_{t}\right)\mid\F_{t-1}\right]\le1$.
Therefore 
\begin{align*}
\E\left[\exp\left(S_{t}\right)\mid\F_{t-1}\right] & =\exp\left(S_{t-1}\right)\E\left[\exp\left(Z_{t}\right)\mid\F_{t-1}\right]\le\exp\left(S_{t-1}\right)
\end{align*}
 which means $(S_{t})_{t\ge1}$ is a supermartingale. By Ville's inequality,
we have, for all $k\ge1$ 
\begin{align*}
\Pr\left[S_{k}\ge\log\frac{1}{\delta}\right] & \le\delta\E\left[\exp\left(S_{1}\right)\right]\le\delta
\end{align*}
In other words, with probability at least $1-\delta$, for all $k\ge1$
\begin{align*}
\sum_{t=1}^{k}Z_{t} & \le\log\frac{1}{\delta}
\end{align*}
Plugging in the definition of $Z_{t}$ we have
\begin{align*}
 & \sum_{t=1}^{k}z_{t}\eta_{t}\Delta_{t+1}+\sum_{t=1}^{k}\left(z_{t}\breg\left(x^{*},x_{t+1}\right)-z_{t}\breg\left(x^{*},x_{t}\right)\right)\\
\le & \log\frac{1}{\delta}+\sum_{t=1}^{k}z_{t}\eta_{t}\left\langle x^{*}-x_{t},\theta_{t}^{b}\right\rangle +2\sum_{t=1}^{k}z_{t}\eta_{t}^{2}\left\Vert \theta_{t}^{b}\right\Vert _{*}^{2}\\
 & +\sum_{t=1}^{k}\left(\left(2z_{t}\eta_{t}^{2}+\frac{3}{8\lambda_{t}^{2}}+24z_{t}^{2}\eta_{t}^{4}\lambda_{t}^{2}\right)\E\left[\left\Vert \theta_{t}^{u}\right\Vert _{*}^{2}\mid\F_{t-1}\right]\right)
\end{align*}
Note that we have $z_{t}$ is a decreasing sequence, hence the $\mbox{LHS}$
of the above inequality can be bounded by
\begin{align*}
\mbox{LHS} & =\sum_{t=1}^{k}z_{t}\eta_{t}\Delta_{t+1}+z_{k}\breg\left(x^{*},x_{k+1}\right)-z_{1}\breg\left(x^{*},x_{1}\right)+\sum_{t=2}^{k}\left(z_{k-1}-z_{k}\right)\breg\left(x^{*},x_{k}\right)\\
 & \ge\sum_{t=1}^{k}z_{t}\eta_{t}\Delta_{t+1}+z_{k}\breg\left(x^{*},x_{k+1}\right)-z_{1}\breg\left(x^{*},x_{1}\right)
\end{align*}
We obtain from here the desired inequality.
\end{proof}

\begin{prop}
\label{prop:convex-general-choice}We assume that the event $E(\delta)$
happens. Suppose that for some $\ell\le T$, there are constants $C_{1}$
and $C_{2}$ such that for all $t\le\ell$

1. $\lambda_{t}\eta_{t}=C_{1}$

2. $\sum_{t=1}^{\ell}\left(\frac{1}{\lambda_{t}}\right)^{p}\le C_{2}$

3. $\left(\frac{1}{\lambda_{t}}\right)^{2p}\le C_{3}\left(\frac{1}{\lambda_{t}}\right)^{p}$

4. $\left\Vert \n(x_{t})\right\Vert _{*}\le\frac{\lambda_{t}}{2}$

Then for all $t\le\ell+1$
\begin{align*}
\sum_{i=1}^{t}\eta_{i}\Delta_{i+1}+\breg\left(x^{*},x_{t+1}\right) & \le\frac{1}{2}\left(R_{1}+8AC_{1}\right)^{2}
\end{align*}
for $A\ge\max\left\{ \log\frac{1}{\delta}+26\sigma^{p}C_{2}+\frac{2\sigma^{2p}C_{2}C_{3}}{A};1\right\} $.
\end{prop}
\begin{proof}
We will prove by induction that on $k$
\begin{align*}
\sum_{i=1}^{k}\eta_{i}\Delta_{i+1}+\breg\left(x^{*},x_{k+1}\right) & \le\frac{1}{2}\left(R_{1}+8AC_{1}\right)^{2}.
\end{align*}
The base case $k=0$ is trivial. We have $\breg\left(x^{*},x_{1}\right)=\frac{R_{1}^{2}}{2}$.
Suppose the statement is true for all $t\le k\le\ell$. Now we show
for $k+1$. Recall that 
\begin{align*}
z_{t} & =\frac{1}{2\eta_{t}\lambda_{t}\max_{i\le t}\sqrt{2\breg\left(x^{*},x_{i}\right)}+16Q\eta_{t}^{2}\lambda_{t}^{2}}
\end{align*}
Let us choose $Q=A>1$. By the induction hypothesis and the assumption
that $\lambda_{t}\eta_{t}=C_{1}$, we have
\begin{align*}
z_{t} & \le\frac{1}{2C_{1}\left(R_{1}+8AC_{1}\right)}\\
z_{k} & \ge\frac{1}{2\eta_{k}\lambda_{k}\left(R_{1}+8AC_{1}\right)+16A\eta_{k}^{2}\lambda_{k}^{2}}\\
 & =\frac{1}{2C_{1}\left(R_{1}+16AC_{1}\right)}
\end{align*}
Since $z_{k}$ is a decreasing sequence 
\begin{align*}
z_{k}\sum_{t=1}^{k}\eta_{t}\Delta_{t+1}+z_{k}\breg\left(x^{*},x_{k+1}\right) & \le z_{1}\breg\left(x^{*},x_{1}\right)+\log\frac{1}{\delta}+\sum_{t=1}^{k}z_{t}\eta_{t}\left\langle x^{*}-x_{t},\theta_{t}^{b}\right\rangle +2\sum_{t=1}^{k}z_{t}\eta_{t}^{2}\left\Vert \theta_{t}^{b}\right\Vert _{*}^{2}\\
 & +\sum_{t=1}^{k}\left(\left(2z_{t}\eta_{t}^{2}+\frac{3}{8\lambda_{t}^{2}}+24z_{t}^{2}\eta_{t}^{4}\lambda_{t}^{2}\right)\E\left[\left\Vert \theta_{t}^{u}\right\Vert _{*}^{2}\mid\F_{t-1}\right]\right)
\end{align*}
By the choice of $\lambda_{t}$, for all $t\le k$, $\left\Vert \n(x_{t})\right\Vert _{*}\le\frac{\lambda_{t}}{2}$,
we can apply Lemma \ref{lem:bias-bounds} and have 
\begin{align*}
\left\Vert \theta_{t}^{b}\right\Vert _{*} & \le4\sigma^{p}\lambda_{t}^{1-p};\\
\E\left[\left\Vert \theta_{t}^{u}\right\Vert _{*}^{2}\mid\F_{t-1}\right] & \le40\sigma^{p}\lambda_{t}^{2-p}.
\end{align*}
Thus we have
\begin{align*}
 & z_{k}\sum_{t=1}^{k}\eta_{t}\Delta_{t+1}+z_{k}\breg\left(x^{*},x_{k+1}\right)\\
\le & z_{1}\breg\left(x^{*},x_{1}\right)+\log\frac{1}{\delta}+4\sum_{t=1}^{k}z_{t}\eta_{t}\sigma^{p}\lambda_{t}^{1-p}\sqrt{2\breg\left(x^{*},x_{t}\right)}+32\sum_{t=1}^{k}z_{t}\eta_{t}^{2}\sigma^{2p}\lambda_{t}^{2-2p}\\
 & +40\sum_{t=1}^{k}\left(\left(2z_{t}\eta_{t}^{2}+\frac{3}{8\lambda_{t}^{2}}+24z_{t}^{2}\eta_{t}^{4}\lambda_{t}^{2}\right)\sigma^{p}\lambda_{t}^{2-p}\right)\\
\le & z_{1}\breg\left(x^{*},x_{1}\right)+\log\frac{1}{\delta}+\frac{2C_{1}\left(R_{1}+8AC_{1}\right)\sigma^{p}}{C_{1}\left(R_{1}+8AC_{1}\right)}\sum_{t=1}^{k}\left(\frac{1}{\lambda_{t}}\right)^{p}+\frac{16C_{1}^{2}\sigma^{2p}}{C_{1}\left(R_{1}+8AC_{1}\right)}\sum_{t=1}^{k}\left(\frac{1}{\lambda_{t}}\right)^{2p}\\
 & +40\left(\frac{C_{1}^{2}}{C_{1}\left(R_{1}+8AC_{1}\right)}+\frac{3}{8}+\frac{6C_{1}^{4}}{C_{1}^{2}\left(R_{1}+8AC_{1}\right)^{2}}\right)\sigma^{p}\sum_{t=1}^{k}\left(\frac{1}{\lambda_{t}}\right)^{p}\\
\le & \frac{R_{1}^{2}}{4\left(C_{1}R_{1}+8AC_{1}^{2}\right)}+\log\frac{1}{\delta}+2\sigma^{p}C_{2}+\frac{2\sigma^{2p}C_{2}C_{3}}{A}+24\sigma^{p}C_{2}\\
\le & \frac{R_{1}^{2}}{4\left(C_{1}R_{1}+8AC_{1}^{2}\right)}+A
\end{align*}
where for the last inequality we use $\sum_{t=1}^{k}\left(\frac{1}{\lambda_{t}}\right)^{p}\le C_{2}$
and $\left(\frac{1}{\lambda_{t}}\right)^{2p}\le C_{3}\left(\frac{1}{\lambda_{t}}\right)^{p}$.
We obtain
\begin{align*}
\sum_{t=1}^{k}\eta_{t}\Delta_{t+1}+\breg\left(x^{*},x_{k+1}\right) & \le2C_{1}\left(R_{1}+16AC_{1}\right)\left(\frac{R_{1}^{2}}{4\left(C_{1}R_{1}+8AC_{1}^{2}\right)}+A\right)\\
 & =\frac{1}{2}R_{1}^{2}+\frac{4AC_{1}^{2}R_{1}^{2}}{C_{1}R_{1}+8AC_{1}^{2}}+2A\left(C_{1}R_{1}+16AC_{1}^{2}\right)\\
 & \le\frac{1}{2}R_{1}^{2}+6AC_{1}R_{1}+32A^{2}C_{1}^{2}\\
 & \le\frac{1}{2}\left(R_{1}+8AC_{1}\right)^{2}.
\end{align*}
\end{proof}

Now we give the proof of Theorem \ref{thm:convex-convergence}, which
is a direct consequence of Proposition \ref{prop:convex-general-choice}.

\begin{proof}
1. Note that $\eta\le\frac{R_{1}}{16}\frac{1}{4LR_{1}}\le\frac{1}{4L}$.
We have that with probability at least $1-\delta$, event $E(\delta)$
happens. Conditioning on this event, in \ref{prop:convex-general-choice}
we choose 
\[
C_{1}=\frac{R_{1}}{24\gamma};\quad C_{2}=\frac{\gamma}{26\sigma^{p}};\quad C_{3}=\frac{\gamma}{26T\sigma^{p}};\quad A=3\gamma
\]
We have 
\begin{align*}
\lambda_{t}\eta_{t} & =C_{1}\\
\sum_{t=1}^{T}\left(\frac{1}{\lambda_{t}}\right)^{p} & \le\sum_{t=1}^{T}\left(\frac{\gamma}{26T}\right)\frac{1}{\sigma^{p}}=C_{2}\\
\left(\frac{1}{\lambda_{t}}\right)^{2p} & \le\frac{1}{\sigma^{p}}\left(\frac{\gamma}{26T}\right)\left(\frac{1}{\lambda_{t}}\right)^{p}=C_{3}\left(\frac{1}{\lambda_{t}}\right)^{p}\\
\max\left\{ \log\frac{1}{\delta}+26\sigma^{p}C_{2}+\frac{2\sigma^{2p}C_{2}C_{3}}{A};1\right\}  & \le3\gamma=A
\end{align*}
We only need to show that for all $t$
\begin{align*}
\left\Vert \n(x_{t})\right\Vert _{*} & \le\frac{\lambda_{t}}{2}
\end{align*}
We will show this by induction. Indeed, we have 
\begin{align*}
\left\Vert \n(x_{1})\right\Vert _{*} & =\left\Vert \n(x_{1})-\n(x^{*})\right\Vert _{*}+\left\Vert \n(x_{0})-\n(x^{*})\right\Vert _{*}\\
 & +\left\Vert \n(x_{0})-g_{0}\right\Vert _{*}+\left\Vert g_{0}\right\Vert _{*}\\
 & \le L\left\Vert x_{1}-x^{*}\right\Vert +L\left\Vert x_{0}-x^{*}\right\Vert +\mu\sigma+\left\Vert g_{0}\right\Vert _{*}\\
 & \le LR_{1}+LR_{0}+\mu\sigma+\left\Vert g_{0}\right\Vert _{*}\le\frac{\lambda_{1}}{2}
\end{align*}
Suppose that it is true for all $t\le k$. We prove that 
\begin{align*}
\left\Vert \n(x_{k+1})\right\Vert _{*} & \le\frac{\lambda_{k+1}}{2}
\end{align*}
By \ref{prop:convex-general-choice} we have 
\begin{align*}
\left\Vert x_{k+1}-x^{*}\right\Vert  & \le\sqrt{2\breg\left(x^{*},x_{k+1}\right)}\le R_{1}+8AC_{1}=2R_{1}
\end{align*}
Thus 
\begin{align*}
\left\Vert \n(x_{k+1})\right\Vert _{*} & =\left\Vert \n(x_{k+1})-\n(x^{*})\right\Vert _{*}+\left\Vert \n(x_{0})-\n(x^{*})\right\Vert _{*}\\
 & +\left\Vert \n(x_{0})-g_{0}\right\Vert _{*}+\left\Vert g_{0}\right\Vert _{*}\\
 & \le L\left\Vert x_{k+1}-x^{*}\right\Vert +L\left\Vert x_{0}-x^{*}\right\Vert +\mu\sigma+\left\Vert g_{0}\right\Vert _{*}\\
 & \le2LR_{1}+LR_{0}+\mu\sigma+\left\Vert g_{0}\right\Vert _{*}\le\frac{\lambda_{k+1}}{2}
\end{align*}
as needed. Therefore from Lemma \ref{lem:convex-key-inequality} we
have 
\begin{align*}
\eta\sum_{t=1}^{T}\Delta_{t+1}+\breg\left(x^{*},x_{T+1}\right) & \le2R_{1}^{2}
\end{align*}
which gives
\begin{align*}
\frac{1}{T}\sum_{t=2}^{T+1}\Delta_{t} & \le\frac{2R_{1}^{2}}{\eta}=48R_{1}\max\left\{ 26^{\frac{1}{p}}T^{\frac{1-p}{p}}\sigma\gamma^{\frac{p-1}{p}};2\left(2LR_{1}+LR_{0}+\mu\sigma+\left\Vert g_{0}\right\Vert _{*}\right)T^{-1}\gamma\right\} 
\end{align*}

2. We can follow the similar steps. Notice that $\left(\eta_{t}\right)$
is a decreasing sequence. We also use fact \ref{fact:fac} to verify
the second condition of Proposition \ref{prop:convex-general-choice}.
The proof is omitted.
\end{proof}

\begin{fact}
\label{fact:fac}We have $\sum_{t=1}^{\infty}\frac{1}{2t\left(1+\log t\right)^{2}}<1.$
\end{fact}
\begin{rem}
In Theorem \ref{thm:convex-convergence}, we use the initial distance
$R_{1}$ to the optimal solution to set the step size and clipping
parameters. This information is generally not available, but can be
avoided. For example, for constrained problems in which we know that
the domain radius is bounded by $R$, we can replace $R_{1}$ in Theorem
\ref{thm:convex-convergence} by $R$ without change in the dependency.
For the general problem, the choice of parameters in Theorem \ref{thm:convex-remove-constants}
does not require knowledge of any constants $T,\sigma,\delta$ or
$R_{1}$. We need, however, a mild assumption on knowing an upper
bound $\nabla_{1}$ of $\left\Vert \n(x_{1})\right\Vert _{*}$ which
can be estimated with good accuracy (see Assumption 5 for a discussion).
We also note that when $\sigma$ is unknown, the convergence guarantee
loses the adaptivity to this constant.
\end{rem}
\begin{thm}
\label{thm:convex-remove-constants}Assume that $f$ satisfies Assumption
(1), (2), (3), (4) and (5). Let $\gamma=\max\left\{ \log\frac{1}{\delta};1\right\} $;
$R_{1}=\sqrt{2\breg\left(x^{*},x_{1}\right)}$ and assume that $\nabla_{1}$
is an upper bound of $\left\Vert \n(x_{1})\right\Vert _{*}$. We choose
$\lambda_{t}$ and $\eta_{t}$ such that 
\begin{align*}
\lambda_{t} & =\max\left\{ \left(52t(1+\log t)^{2}c_{2}\right)^{1/p};2\left(L\max_{i\le t}\left\Vert x_{i}-x_{1}\right\Vert +\nabla_{1}\right);\frac{Lc_{1}}{6}\right\} \text{, and }\\
\eta_{t} & =\frac{c_{1}}{24\lambda_{t}}=\frac{c_{1}}{24}\min\left\{ \left(52t(1+\log t)^{2}c_{2}\right)^{-1/p};\frac{1}{2\left(L\max_{i\le t}\left\Vert x_{i}-x_{1}\right\Vert +\nabla_{1}\right)};\frac{6}{Lc_{1}}\right\} ,
\end{align*}
where the constants $c_{1}$ and $c_{2}$ are to ensure the correctness
of the dimensions. Then with probability at least $1-\delta$ we have
\begin{align*}
\frac{1}{T}\sum_{t=2}^{T+1}\Delta_{t} & \le\frac{8}{Tc_{1}}\left(R_{1}+\frac{c_{1}}{3}\left(\gamma+\frac{2\sigma^{p}}{c_{2}}\right)\right)^{2}\max\left\{ \left(52T(1+\log T)^{2}c_{2}\right)^{1/p};4R_{1}L+\frac{2c_{1}}{3}L\left(\gamma+\frac{2\sigma^{p}}{c_{2}}\right)+2\nabla_{1};\frac{Lc_{1}}{6}\right\} \\
 & =\widetilde{O}\left(T^{\frac{1-p}{p}}\right).
\end{align*}
\end{thm}
\begin{proof}
Note that $\eta_{t}\le\frac{1}{4L}$. We have that with probability
at least $1-\delta$, event $E(\delta)$ happens. Conditioning on
this event, in \ref{prop:convex-general-choice}. We choose 
\[
C_{1}=\frac{c_{1}}{24};\quad C_{2}=\frac{1}{26c_{2}};\quad C_{3}=\frac{1}{52c_{2}};\quad A=\gamma+\frac{2\sigma^{p}}{c_{2}}
\]
We verify the conditions of Proposition \ref{prop:convex-general-choice}
\begin{align*}
\lambda_{t}\eta_{t} & =C_{1}\\
\sum_{t=1}^{T}\left(\frac{1}{\lambda_{t}}\right)^{p} & \le\sum_{t=1}^{T}\frac{1}{52t(1+\log t)^{2}c_{2}}\le\frac{1}{26c_{2}}=C_{2}\\
\left(\frac{1}{\lambda_{t}}\right)^{2p} & \le\frac{1}{52tc_{2}}\left(\frac{1}{\lambda_{t}}\right)^{p}\le C_{3}\left(\frac{1}{\lambda_{t}}\right)^{p}\\
\max\left\{ \log\frac{1}{\delta}+26\sigma^{p}C_{2}+\frac{2\sigma^{2p}C_{2}C_{3}}{A};1\right\}  & =\max\left\{ \log\frac{1}{\delta}+\frac{\sigma^{p}}{c_{2}}+\frac{\sigma^{p}}{c_{2}};1\right\} \le A
\end{align*}
where we have $\frac{2\sigma^{2p}C_{2}C_{3}}{A}\le2\sigma^{2p}C_{2}C_{3}\times\frac{c_{2}}{2\sigma^{p}}\le\frac{\sigma^{p}}{c_{2}}$.
\begin{align*}
\left\Vert \n(x_{t})\right\Vert _{*} & =\left\Vert \n(x_{t})-\n(x_{1})\right\Vert _{*}+\left\Vert \n(x_{1})\right\Vert _{*}\\
 & \le L\left\Vert x_{t}-x_{1}\right\Vert _{*}+\left\Vert \n(x_{1})\right\Vert _{*}\le\frac{\lambda_{k+1}}{2}
\end{align*}
Therefore from Lemma \ref{lem:convex-key-inequality} we have 
\begin{align*}
\eta_{T}\sum_{t=1}^{T}\Delta_{t+1}+\breg\left(x^{*},x_{T+1}\right) & \le\frac{1}{2}\left(R_{1}+8AC_{1}\right)^{2}\\
 & =\frac{1}{2}\left(R_{1}+\frac{c_{1}}{3}\left(\gamma+\frac{2\sigma^{p}}{c_{2}}\right)\right)^{2}
\end{align*}
which gives
\begin{align*}
\frac{1}{T}\sum_{t=2}^{T+1}\Delta_{t} & \le\frac{1}{2T\eta_{T}}\left(R_{1}+\frac{c_{1}}{3}\left(\gamma+\frac{2\sigma^{p}}{c_{2}}\right)\right)^{2}\\
 & =\frac{8}{Tc_{1}}\left(R_{1}+\frac{c_{1}}{3}\left(\gamma+\frac{2\sigma^{p}}{c_{2}}\right)\right)^{2}\max\left\{ \left(52T(1+\log T)^{2}c_{2}\right)^{1/p};2\left(L\max_{i\le T}\left\Vert x_{i}-x_{1}\right\Vert +\nabla_{1}\right);\frac{L}{8}\right\} 
\end{align*}
Note that 
\begin{align*}
\left\Vert x_{i}-x_{1}\right\Vert  & \le\left\Vert x_{i}-x^{*}\right\Vert +\left\Vert x_{1}-x^{*}\right\Vert \\
 & \le2R_{1}+\frac{c_{1}}{3}\left(\gamma+\frac{2\sigma^{p}}{c_{2}}\right)
\end{align*}
which gives us the final convergence rate.
\end{proof}

\section{Clipped Accelerated Stochastic Mirror Descent \label{sec:Clipped-Accelerated-Stochastic}}

\begin{algorithm}
\caption{Clipped-ASMD}
\label{alg:clipped-asmd}

Parameters: initial point $y_{1}=z_{1}$, step sizes $\left\{ \eta_{t}\right\} $,
clipping parameters $\left\{ \lambda_{t}\right\} $, $\psi$ is $1$-strongly
convex wrt $\left\Vert \cdot\right\Vert $

for $t=1$ to $T$ do

$\quad$Set $\alpha_{t}=\frac{2}{t+1}$

$\quad$$x_{t}=\left(1-\alpha_{t}\right)y_{t}+\alpha_{t}z_{t}$

$\quad$$\tn(x_{t})=\min\left\{ 1,\frac{\lambda_{t}}{\left\Vert \hn(x_{t})\right\Vert _{*}}\right\} \hn(x_{t})$

$\quad$$z_{t+1}=\arg\min_{x\in\dom}\left\{ \eta_{t}\left\langle \tn(x_{t}),x\right\rangle +\breg\left(x,z_{t}\right)\right\} $

$\quad$$y_{t+1}=\left(1-\alpha_{t}\right)y_{t}+\alpha_{t}z_{t+1}$
\end{algorithm}

In this section, we extend the analysis of Clipped-SMD to the case
of Clipped Accelerated Stochastic Mirror Descent (Algorithm \ref{alg:clipped-asmd}).
We will see that the analysis is basically the same with little modification.
We present in Algorithm \ref{alg:clipped-asmd} the clipped version
of accelerated stochastic mirror descent (see \cite{lan2020first}),
where the clipped gradient $\tn(x_{t})$ is used to update the iterate
instead of the stochastic gradient $\hn(x_{t})$.

We use the following additional assumption:

\textbf{(5') Global minimizer}: We assume that $\nabla f(x^{*})=0$.

In words, we assume that the global minimizer lies in the domain of
the problem. This assumption is consistent with the works of \cite{gorbunov2020stochastic,sadiev2023high}.

We first provide the convergence guarantee for known time horizon.
The statement for unknown $T$ is deferred to the appendix.
\begin{thm}
\label{thm:convex-accelerated-main-convergence}Assume that $f$ satisfies
Assumption (1), (2), (3), (4) and (5'). Let $\gamma=\max\left\{ \log\frac{1}{\delta};1\right\} $;
and $R_{1}=\sqrt{2\breg\left(x^{*},x_{1}\right)}$. For known $T$,
we choose a constant $c$ and $\lambda_{t}$ and $\eta_{t}$ such
that
\begin{align*}
c & =\max\left\{ 10^{4};\frac{4\left(T+1\right)\left(\frac{26T}{\gamma}\right)^{\frac{1}{p}}\sigma}{\gamma LR_{1}}\right\} \\
\lambda_{t} & =\frac{cR_{1}\gamma L\alpha_{t}}{8}=\max\left\{ \frac{10^{4}R_{1}\gamma L}{6(t+1)};\frac{T+1}{t+1}\left(\frac{26T}{\gamma}\right)^{1/p}\sigma\right\} \\
\eta_{t} & =\frac{1}{3c\gamma^{2}L\alpha_{t}}=\frac{R_{1}}{24\gamma}\min\left\{ \frac{4(t+1)}{10^{4}R_{1}\gamma L};\frac{t+1}{T+1}\left(\frac{26T}{\gamma}\right)^{-1/p}\sigma^{-1}\right\} 
\end{align*}
Then with probability at least $1-\delta$
\begin{align*}
f\left(y_{T+1}\right)-f\left(x^{*}\right) & \le6\max\left\{ 10^{4}L\gamma^{2}R_{1}^{2}(T+1)^{-2};4R_{1}\left(T+1\right)^{-1}\left(26T\right)^{\frac{1}{p}}\gamma^{\frac{p-1}{p}}\sigma\right\} 
\end{align*}
\end{thm}
\begin{rem}
One feature of the accelerated algorithm is the interpolation between
the two regimes: When $\sigma$ is large, the algorithm achieves the
$O\left(T^{\frac{1-p}{p}}\right)$ convergence, same as the unaccelerated
algorithm; however, when $\sigma$ is sufficiently small, the algorithm
achieves the accelerated $O\left(T^{-2}\right)$ rate.
\end{rem}
We also start with the basic analysis of accelerated stochastic mirror
descent in the following lemma.
\begin{lem}
\label{lem:convex-acc-basic-inequality}Assume that $f$ satisfies
Assumption (1), (2), (3), (4) and $\eta_{t}\le\frac{1}{2L\alpha_{t}}$,
the iterate sequence $(x_{t})_{t\ge1}$ output by Algorithm \ref{alg:clipped-smd}
satisfies the following
\begin{align*}
 & \frac{\eta_{t}}{\alpha_{t}}\left(f\left(y_{t+1}\right)-f\left(x^{*}\right)\right)-\frac{\eta_{t}\left(1-\alpha_{t}\right)}{\alpha_{t}}\left(f\left(y_{t}\right)-f\left(x^{*}\right)\right)+\breg\left(x^{*},z_{t+1}\right)-\breg\left(x^{*},z_{t}\right)\\
\le & \eta_{t}\left\langle \theta_{t}^{u},x^{*}-z_{t}\right\rangle +\eta_{t}\left\langle \theta_{t}^{b},x^{*}-z_{t}\right\rangle +2\eta_{t}^{2}\left(\left\Vert \theta_{t}^{u}\right\Vert _{*}^{2}-\E\left[\left\Vert \theta_{t}^{u}\right\Vert _{*}^{2}\mid\F_{t-1}\right]\right)+2\eta_{t}^{2}\left\Vert \theta_{t}^{b}\right\Vert _{*}^{2}+2\eta_{t}^{2}\E\left[\left\Vert \theta_{t}^{u}\right\Vert _{*}^{2}\mid\F_{t-1}\right]
\end{align*}
\end{lem}
Similarly to the previous section, we define the following variables
\begin{align*}
Z_{t} & =z_{t}\Bigg(\frac{\eta_{t}}{\alpha_{t}}\left(f\left(y_{t+1}\right)-f\left(x^{*}\right)\right)-\frac{\eta_{t}\left(1-\alpha_{t}\right)}{\alpha_{t}}\left(f\left(y_{t}\right)-f\left(x^{*}\right)\right)+\breg\left(x^{*},z_{t+1}\right)-\breg\left(x^{*},z_{t}\right)\\
 & \qquad-\eta_{t}\left\langle \theta_{t}^{b},x^{*}-z_{t}\right\rangle -2\eta_{t}^{2}\left\Vert \theta_{t}^{b}\right\Vert _{*}^{2}-2\eta_{t}^{2}\E\left[\left\Vert \theta_{t}^{u}\right\Vert _{*}^{2}\mid\F_{t-1}\right]\Bigg)-\left(\frac{3}{8\lambda_{t}^{2}}+24z_{t}^{2}\eta_{t}^{4}\lambda_{t}^{2}\right)\E\left[\left\Vert \theta_{t}^{u}\right\Vert ^{2}\mid\F_{t-1}\right]\\
\mbox{where }z_{t} & =\frac{1}{2\eta_{t}\lambda_{t}\max_{i\le t}\sqrt{2\breg\left(x^{*},x_{i}\right)}+16Q\eta_{t}^{2}\lambda_{t}^{2}}
\end{align*}
for a constant $Q\ge1$ and
\begin{align*}
S_{t} & =\sum_{i=1}^{t}Z_{i}
\end{align*}
Following the same analysis, we obtain the claims in Lemma \ref{lem:convex-accelerated-key-inequality}
and Proposition \ref{prop:convex-acceleration-general-choice} for
which we will omit the proofs. The only step we need to pay attention
to when showing Lemma \ref{lem:convex-accelerated-key-inequality}
is when we bound the sum 
\[
\sum_{t=1}^{k}\frac{z_{t}\eta_{t}}{\alpha_{t}}\left(f\left(y_{t+1}\right)-f\left(x^{*}\right)\right)-\frac{z_{t}\eta_{t}\left(1-\alpha_{t}\right)}{\alpha_{t}}\left(f\left(y_{t}\right)-f\left(x^{*}\right)\right)
\]
Notice that if we assume $\frac{\eta_{t-1}}{\alpha_{t-1}}\ge\frac{\eta_{t}\left(1-\alpha_{t}\right)}{\alpha_{t}}$,
since $z_{t}$ is a decreasing sequence and $\alpha_{1}=0$, we can
lower bound the above sum by the last term $\frac{z_{k}\eta_{k}}{\alpha_{k}}\left(f\left(y_{k+1}\right)-f\left(x^{*}\right)\right)$,
which gives us the desired inequality.
\begin{lem}
\label{lem:convex-accelerated-key-inequality}Assume that for all
$t\ge1$, $\eta_{t}$ satisfies $\frac{\eta_{t-1}}{\alpha_{t-1}}\ge\frac{\eta_{t}\left(1-\alpha_{t}\right)}{\alpha_{t}}$.
For any $\delta>0$, let $E(\delta)$ be the event that for all $1\le k\le T$
\begin{align*}
 & \frac{z_{k}\eta_{k}}{\alpha_{k}}\left(f\left(y_{k+1}\right)-f\left(x^{*}\right)\right)+z_{k}\breg\left(x^{*},x_{k+1}\right)\\
 & \text{\ensuremath{\le}}z_{1}\breg\left(x^{*},x_{1}\right)+\log\frac{1}{\delta}+\sum_{t=1}^{k}z_{t}\eta_{t}\left\langle x^{*}-x_{t},\theta_{t}^{b}\right\rangle +2\sum_{t=1}^{k}z_{t}\eta_{t}^{2}\left\Vert \theta_{t}^{b}\right\Vert _{*}^{2}\\
 & \quad+\sum_{t=1}^{k}\left(\left(2z_{t}\eta_{t}^{2}+\frac{3}{8\lambda_{t}^{2}}+24z_{t}^{2}\eta_{t}^{4}\lambda_{t}^{2}\right)\E\left[\left\Vert \theta_{t}^{u}\right\Vert _{*}^{2}\mid\F_{t-1}\right]\right)
\end{align*}
Then $\Pr\left[E(\delta)\right]\ge1-\delta$.
\end{lem}
\begin{prop}
\label{prop:convex-acceleration-general-choice}We assume that the
event $E(\delta)$ happens. Suppose that for some $\ell\le T$, there
are constants $C_{1}$ and $C_{2}$ such that for all $t\le\ell$

1. $\lambda_{t}\eta_{t}=C_{1}$

2. $\sum_{t=1}^{\ell}\left(\frac{1}{\lambda_{t}}\right)^{p}\le C_{2}$

3. $\left(\frac{1}{\lambda_{t}}\right)^{2p}\le C_{3}\left(\frac{1}{\lambda_{t}}\right)^{p}$

4. $\left\Vert \n(x_{t})\right\Vert _{*}\le\frac{\lambda_{t}}{2}$

Then for all $t\le\ell+1$
\begin{align*}
\frac{\eta_{t}}{\alpha_{t}}\left(f\left(y_{t+1}\right)-f\left(x^{*}\right)\right)+\breg\left(x^{*},z_{t+1}\right) & \le\frac{1}{2}\left(R_{1}+8AC_{1}\right)^{2}
\end{align*}
for $A\ge\max\left\{ \log\frac{1}{\delta}+26\sigma^{p}C_{2}+\frac{2\sigma^{2p}C_{2}C_{3}}{A};1\right\} $
\end{prop}
We are ready to prove Theorem \ref{thm:convex-accelerated-main-convergence}.

\begin{proof}
Note that $\eta_{t}\le\frac{1}{2c\gamma^{2}L\alpha_{t}}\le\frac{1}{2L\alpha_{t}}$
and
\begin{align*}
\frac{\eta_{t-1}}{\alpha_{t-1}} & =\frac{t^{2}}{8c\gamma^{2}L}\\
\frac{\eta_{t}\left(1-\alpha_{t}\right)}{\alpha_{t}} & =\frac{(t+1)(t-1)}{8c\gamma^{2}L}
\end{align*}
thus $\frac{\eta_{t-1}}{\alpha_{t-1}}\ge\frac{\eta_{t}\left(1-\alpha_{t}\right)}{\alpha_{t}}$.
We have that with probability at least $1-\delta$, event $E(\delta)$
happens. Conditioning on this event, in \ref{prop:convex-general-choice}
We choose 
\[
C_{1}=\frac{R_{1}}{24\gamma};\quad C_{2}=\frac{\gamma}{26\sigma^{p}};\quad C_{3}=\frac{\gamma}{26T\sigma^{p}};\quad A=3\gamma
\]
We verify the conditions of Proposition \ref{prop:convex-acceleration-general-choice}.
The first four conditions hold similarly as in previous section. 

We will show by induction that for all $t\ge1$, $\left\Vert \n(x_{t})\right\Vert _{*}\le\frac{\lambda_{t}}{2}$
and $\max\left\{ \left\Vert x_{t}-x^{*}\right\Vert ,\left\Vert y_{t}-x^{*}\right\Vert ,\left\Vert z_{t}-x^{*}\right\Vert \right\} \le2R_{1}$.

For $t=1$ notice that $x_{1}=y_{1}=z_{1}$ thus we have 
\begin{align*}
\left\Vert \n(x_{1})\right\Vert _{*} & =\left\Vert \n(x_{1})-\n(x^{*})\right\Vert _{*}\le LR_{1}\le\frac{\lambda_{1}}{2}
\end{align*}
Now assume that the claim holds for $1\le t\le k$, we show for $t=k+1$.
By Proposition \ref{prop:convex-acceleration-general-choice} we know
that 
\begin{align*}
\frac{2\eta_{k}}{\alpha_{k}}f\left(y_{k+1}\right)-f\left(x^{*}\right)+\left\Vert z_{k+1}-x^{*}\right\Vert ^{2} & \le4R_{1}^{2}
\end{align*}
Furthermore 
\begin{align*}
\left\Vert y_{k+1}-x^{*}\right\Vert  & \le\left(1-\alpha_{k}\right)\left\Vert y_{k}-x^{*}\right\Vert +\alpha_{k}\left\Vert z_{k+1}-x^{*}\right\Vert \le2R_{1}\\
\left\Vert x_{k+1}-x^{*}\right\Vert  & \le\left(1-\alpha_{k}\right)\left\Vert y_{k+1}-x^{*}\right\Vert +\alpha_{k}\left\Vert z_{k+1}-x^{*}\right\Vert \le2R_{1}
\end{align*}
For $k\ge1$ we have $\alpha_{k+1}=\frac{2}{k+2}<1$; $\frac{\alpha_{k+1}}{1-\alpha_{k+1}}=\frac{2}{k}\le\frac{4}{k+2}\le2\alpha_{t+1}$
and $\alpha_{t}\le\frac{3}{2}\alpha_{t+1}$. Hence
\begin{align*}
\left\Vert \nabla f(x_{k+1})\right\Vert _{*} & \le\left\Vert \nabla f(x_{k+1})-\nabla f(y_{k+1})\right\Vert _{*}+\left\Vert \nabla f(y_{k+1})-\nabla f(x^{*})\right\Vert _{*}\\
 & \le L\left\Vert x_{k+1}-y_{k+1}\right\Vert +\sqrt{2L\left(f\left(y_{k+1}\right)-f\left(x^{*}\right)\right)}\\
 & \le\frac{L\alpha_{k+1}\left\Vert x_{k+1}-z_{k+1}\right\Vert }{1-\alpha_{k+1}}+2R_{1}\sqrt{\frac{L\alpha_{t}}{2\eta_{t}}}\\
 & \le4LR_{1}\frac{\alpha_{k+1}}{1-\alpha_{k+1}}+2\sqrt{\frac{3}{2}c}\gamma R_{1}L\alpha_{t}\\
 & \le8\gamma LR_{1}\alpha_{t+1}+3\sqrt{\frac{3}{2}c}\gamma LR_{1}\alpha_{t+1}\\
 & \le(8+3\sqrt{\frac{3}{2}c})R_{1}\gamma L\alpha_{t+1}\\
 & =\frac{16(8+3\sqrt{\frac{3}{2}c})\lambda_{t+1}}{2c}\le\frac{\lambda_{t+1}}{2}.
\end{align*}
as needed. Therefore we have 
\begin{align*}
\frac{\eta_{T}}{\alpha_{T}}\left(f\left(y_{T+1}\right)-f\left(x^{*}\right)\right)+\breg\left(x^{*},x_{T+1}\right) & \le2R_{1}^{2}
\end{align*}
which gives
\begin{align*}
f\left(y_{T+1}\right)-f\left(x^{*}\right) & \le\frac{2R_{1}^{2}\alpha_{T}}{\eta_{T}}=6R_{1}^{2}c\gamma^{2}L\alpha_{T}^{2}\\
 & =6\max\left\{ 10^{4}L\gamma^{2}R_{1}^{2}(T+1)^{-2};6R_{1}\left(T+1\right)^{-1}\left(26T\right)^{\frac{1}{p}}\gamma^{\frac{p-1}{p}}\sigma\right\} .
\end{align*}
\end{proof}

\section{Clipped Stochastic Gradient Descent for Nonconvex Functions \label{sec:Clipped-Stochastic-Gradient}}

\begin{algorithm}
\caption{Clipped-SGD}
\label{alg:clipped-sgd}

Parameters: initial point $x_{1}$, step sizes $\left\{ \eta_{t}\right\} $,
clipping parameters $\left\{ \lambda_{t}\right\} $

for $t=1$ to $T$ do

$\quad$$\tn(x_{t})=\min\left\{ 1,\frac{\lambda_{t}}{\left\Vert \hn(x_{t})\right\Vert }\right\} \hn(x_{t})$

$\quad$$x_{t+1}=x_{t}-\eta_{t}\tn(x_{t})$
\end{algorithm}

In this section, we study the convergence of Clipped-SGD for nonconvex
functions. In this setting, we only consider the case when the domain
is $\R^{d}$ and the norm is the $\ell_{2}$ norm. The general framework
in Algorithm \ref{alg:clipped-sgd} was proposed by \cite{gorbunov2020stochastic}
(although for convex objectives) and studied in \cite{cutkosky2021high,nguyen2023high,liu2023breaking,sadiev2023high}.
We note that these latter works, although they achieve the nearly-optimal
time dependency of $\widetilde{O}\left(T^{\frac{2-2p}{3p-2}}\right)$
in various settings, they all have the same limitation in using the
concentration inequalities as a blackbox and require a known time
horizon. We improve these works in this aspect. Once again, we will
show that the whitebox method is a powerful method that allows a tight
analysis of convergence in high probability and at the same time generalizes
well to different cases.

We will show the following guarantee for known $T$. The statement
for unknown $T$ is shown in the appendix.
\begin{thm}
\label{thm:nonconvex-convergence}Assume that $f$ satisfies Assumption
(1'), (2), (3), (4). Let $\gamma=\max\left\{ \log\frac{1}{\delta};1\right\} $
and $\Delta_{1}=f(x_{1})-f^{*}$. For known $T$, we choose $\lambda_{t}$
and $\eta_{t}$ such that 
\begin{align*}
\lambda_{t} & =\lambda=\max\left\{ \left(\frac{8\gamma}{\sqrt{L\Delta_{1}}}\right)^{\frac{1}{p-1}}T^{\frac{1}{3p-2}}\sigma^{\frac{p}{p-1}};2\sqrt{90L\Delta_{1}};32^{\frac{1}{p}}\sigma T^{\frac{1}{3p-2}}\right\} \\
\eta_{t} & =\eta=\frac{\sqrt{\Delta_{1}}T^{\frac{1-p}{3p-2}}}{8\lambda\sqrt{L}\gamma}=\frac{\sqrt{\Delta_{1}}}{8\sqrt{L}\gamma}\min\left\{ \left(\frac{8\gamma}{\sqrt{L\Delta_{1}}}\right)^{\frac{-1}{p-1}}T^{\frac{-p}{3p-2}}\sigma^{\frac{-p}{p-1}};\frac{T^{\frac{1-p}{3p-2}}}{2\sqrt{90L\Delta_{1}}};\frac{T^{\frac{-p}{3p-2}}}{32^{1/p}\sigma}\right\} 
\end{align*}
Then with probability at least $1-\delta$
\begin{align*}
\frac{1}{T}\sum_{t=1}^{T}\left\Vert \n(x_{t})\right\Vert ^{2} & \le720\sqrt{\Delta_{1}L}\gamma\max\left\{ \left(\frac{8\gamma}{\sqrt{L\Delta_{1}}}\right)^{\frac{1}{p-1}}T^{\frac{2-2p}{3p-2}}\sigma^{\frac{p}{p-1}};2\sqrt{90L\Delta_{1}}T^{\frac{1-2p}{3p-2}};32^{1/p}\sigma T^{\frac{2-2p}{3p-2}}\right\} =O\left(T^{\frac{2-2p}{3p-2}}\right).
\end{align*}
\end{thm}
The analysis of \ref{alg:clipped-sgd} starts with the following lemma
whose proof is shown in the Appendix.
\begin{lem}
\label{lem:nonconvex-basic-analysis}Assume that the function $f$
satisfies Assumption (1'), (2), (3), (4) and $\eta_{t}\le\frac{1}{L}$
then for all $t\ge1$,
\begin{align*}
\frac{\eta_{t}}{2}\left\Vert \nabla f(x_{t})\right\Vert ^{2} & \le\Delta_{t}-\Delta_{t+1}+\left(L\eta_{t}^{2}-\eta_{t}\right)\left\langle \nabla f(x_{t}),\theta_{t}^{u}\right\rangle +\frac{3\eta_{t}}{2}\left\Vert \theta_{t}^{b}\right\Vert ^{2}\\
 & +L\eta_{t}^{2}\left(\left\Vert \theta_{t}^{u}\right\Vert ^{2}-\E\left[\left\Vert \theta^{u}\right\Vert ^{2}\mid\F_{t-1}\right]\right)+L\eta_{t}^{2}\E\left[\left\Vert \theta_{t}^{u}\right\Vert ^{2}\mid\F_{t-1}\right]
\end{align*}
\end{lem}
Similarly to the convex analysis, in order to overcome the issues
in the induction-type argument in prior works, we define the following
terms

\begin{align*}
Z_{t} & =z_{t}\left(\frac{\eta_{t}}{2}\left\Vert \nabla f(x_{t})\right\Vert ^{2}+\Delta_{t+1}-\Delta_{t}-\frac{3\eta_{t}}{2}\left\Vert \theta_{t}^{b}\right\Vert ^{2}-L\eta_{t}^{2}\E\left[\left\Vert \theta_{t}^{u}\right\Vert ^{2}\mid\F_{t-1}\right]\right)\\
 & -\left(3z_{t}^{2}L\eta_{t}^{2}\Delta_{t}+6L^{2}z_{t}^{2}\eta_{t}^{4}\lambda_{t}^{2}\right)\E\left[\left\Vert \theta_{t}^{u}\right\Vert ^{2}\mid\F_{t-1}\right]\\
\mbox{where }z_{t} & =\frac{1}{2P_{t}\eta_{t}\lambda_{t}\max_{i\le t}\sqrt{2L\Delta_{i}}+8Q_{t}L\eta_{t}^{2}\lambda_{t}^{2}}
\end{align*}
for $P_{t},Q_{t}\in\F_{t-1}\ge1$ and
\begin{align*}
S_{t} & =\sum_{i=1}^{t}Z_{i}
\end{align*}
The following lemma guarantees a relation between all the terms over
all time $t$ with high probability. The proof for this lemma is a
combination of standard techniques and the techniques shown in prior
sections, and it is deferred to the appendix.
\begin{lem}
\label{lem:nonconvex-key-inequality}For any $\delta>0$, let $E(\delta)$
be the event that for all $1\le k\le T$
\begin{align*}
\frac{1}{2}\sum_{t=1}^{k}z_{t}\eta_{t}\left\Vert \nabla f(x_{t})\right\Vert ^{2}+z_{k}\Delta_{k+1} & \le z_{1}\Delta_{1}+\log\frac{1}{\delta}+\sum_{t=1}^{k}\frac{3z_{t}\eta_{t}}{2}\left\Vert \theta_{t}^{b}\right\Vert ^{2}\\
 & +\sum_{t=1}^{k}\left(\left(3z_{t}^{2}L\eta_{t}^{2}\Delta_{t}+6L^{2}z_{t}^{2}\eta_{t}^{4}\lambda_{t}^{2}+z_{t}L\eta_{t}^{2}\right)\E\left[\left\Vert \theta_{t}^{u}\right\Vert ^{2}\mid\F_{t-1}\right]\right)
\end{align*}
Then $\Pr\left[E(\delta)\right]\ge1-\delta$.
\end{lem}
Now we specify the choice of $\eta_{t}$ and $\lambda_{t}$. The following
lemma gives a general condition for the choice of $\eta_{t}$ and
$\lambda_{t}$ that gives the right convergence rate in time $T$.
\begin{prop}
\label{prop:nonconvex-general-choice}We assume that the event $E$
happens. Suppose that for some $\ell\le T$, there are constants $C_{1}$,
$C_{2}$ and $C_{3}$ such that for all $t\le\ell$

1. $\lambda_{t}\eta_{t}\sqrt{2L}\le C_{1}$

2. $\frac{1}{L\eta_{t}}\left(\frac{1}{\lambda_{t}}\right)^{p}\le C_{2}$

3. $\sum_{t=1}^{T}L\left(\frac{1}{\lambda_{t}}\right)^{p}\lambda_{t}^{2}\eta_{t}^{2}\le C_{3}$

4. $\left\Vert \n(x_{t})\right\Vert \le\frac{\lambda_{t}}{2}$

Then for all $t\le\ell+1$
\begin{align*}
\frac{1}{2}\sum_{i=1}^{t}\eta_{i}\left\Vert \nabla f(x_{i})\right\Vert ^{2}+\Delta_{t+1} & \le\left(\sqrt{\Delta_{1}}+2\sqrt{A}C_{1}\right)^{2}
\end{align*}

for a constant $A\ge\max\left\{ 64\left(\log\frac{1}{\delta}+\frac{60\sigma^{p}C_{3}}{C_{1}^{2}}\right)^{2}+\frac{48\sigma^{2p}C_{2}C_{3}+140\sigma^{p}C_{3}}{C_{1}^{2}};1\right\} $.
\end{prop}
Finally we give the proof for Theorem \ref{thm:convex-convergence},
which is a direct consequence of Proposition \ref{prop:convex-general-choice}.

\begin{proof}
Note that $\eta\le\frac{T^{\frac{1-p}{3p-2}}}{16\sqrt{90}L\gamma}\le\frac{1}{L}$.
We have that with probability at least $1-\delta$, event $E(\delta)$
happens. Conditioning on this event, we verify the condition of Lemma
\ref{prop:nonconvex-general-choice}. We select the following constants
\[
C_{1}=\frac{\sqrt{\Delta_{1}}}{4\sqrt{2}\gamma};\quad C_{2}\le\frac{1}{\sigma^{p}};\quad C_{3}\leq\frac{\Delta_{1}}{2048\sigma^{p}\gamma};\quad A=256\gamma^{2}
\]
We verify in Lemma \ref{lem:nonconvex-condition-verification} that
for these choice of constants, conditions (1)-(3) of Proposition \ref{prop:nonconvex-general-choice}
are satisfied. Furthermore, we have
\begin{align*}
 & 64\left(\log\frac{1}{\delta}+\frac{60\sigma^{p}C_{3}}{C_{1}^{2}}\right)^{2}+\frac{48\sigma^{2p}C_{2}C_{3}+140\sigma^{p}C_{3}}{C_{1}^{2}}\\
 & =64\left(\log\frac{1}{\delta}+60\log\frac{1}{\delta}\frac{32}{\Delta_{1}}\frac{\Delta_{1}}{2048}\right)^{2}+\left(48\frac{\Delta_{1}}{2048}+140\frac{\Delta_{1}}{2048}\right)\frac{32}{\Delta_{1}}\\
 & \le256\gamma^{2}=A
\end{align*}
We only need to show that for all $t$, $\left\Vert \n(x_{t})\right\Vert \le\frac{\lambda_{t}}{2}$.
We will show this by induction. Indeed, for the base case we have
$\left\Vert \n(x_{1})\right\Vert =\sqrt{2L\Delta_{1}}\le\frac{\lambda_{1}}{2}$.
Suppose that it is true for all $t\le k$. We will prove that $\left\Vert \n(x_{k+1})\right\Vert \le\frac{\lambda_{k+1}}{2}$.
By Lemma \ref{prop:nonconvex-general-choice} and the induction hypothesis
\begin{align*}
\Delta_{k+1} & \le\left(\sqrt{\Delta_{1}}+2\sqrt{A}C_{1}\right)\le\left(\sqrt{\Delta_{1}}+\frac{\sqrt{\Delta_{1}}}{2\sqrt{2}\gamma}\times16\gamma\right)^{2}\le45\Delta_{1}
\end{align*}
Thus we get
\begin{align*}
\left\Vert \n(x_{k+1})\right\Vert  & =\sqrt{2L\Delta_{k+1}}\le\sqrt{90L\Delta_{1}}\le\frac{\lambda_{k+1}}{2}
\end{align*}
as needed. From Lemma \ref{lem:convex-key-inequality} we have 
\begin{align*}
\frac{\eta}{2}\sum_{t=1}^{T}\left\Vert \n(x_{t})\right\Vert ^{2}+\Delta_{k+1} & \le45\Delta_{1}
\end{align*}
Therefore 
\begin{align*}
\frac{1}{T}\sum_{t=1}^{T}\left\Vert \n(x_{t})\right\Vert ^{2} & \le\frac{90\Delta_{1}}{\eta T}=720\sqrt{\Delta_{1}L}\gamma\max\left\{ \left(\frac{8\gamma}{\sqrt{L\Delta_{1}}}\right)^{\frac{1}{p-1}}T^{\frac{2-2p}{3p-2}}\sigma^{\frac{p}{p-1}};2\sqrt{90L\Delta_{1}}T^{\frac{1-2p}{3p-2}};32^{\frac{1}{p}}\sigma T^{\frac{2-2p}{3p-2}}\right\} .
\end{align*}

\end{proof}

\section{Conclusion}

In this work, we propose a new approach to analyze various clipped
gradient algorithms in the presence of heavy tailed noise. The analysis
can be easily extended to the case of nonsmooth convex objectives.
High probability convergence with heavy tailed noises has been studied
for variational inequalities in \cite{gorbunov2022clipped} with similar
techniques to convex optimization. We leave the question of extending
our method to this setting for future investigation.

\bibliographystyle{plain}
\bibliography{ref}

\appendix

\section{Missing Proofs from Section \ref{sec:Clipped-Stochastic-Mirror}}
\begin{lem}
Suppose that $\eta_{t}\le\frac{1}{4L}$ and assume $f$ satisfies
Assumption (1), (2), (3) as well as the following condition
\begin{align}
f(y)-f(x) & \le\left\langle \n(x),y-x\right\rangle +G\left\Vert y-x\right\Vert +\frac{L}{2}\left\Vert y-x\right\Vert ^{2},\quad\forall y,x\in\dom.\label{eq:1}
\end{align}
Then the iterate sequence $(x_{t})_{t\ge1}$ output by Algorithm \ref{alg:clipped-smd}
satisfies the following:
\begin{align*}
\eta_{t}\Delta_{t+1} & \le\breg\left(x^{*},x_{t}\right)-\breg\left(x^{*},x_{t+1}\right)+\eta_{t}\left\langle x^{*}-x_{t},\theta_{t}^{u}\right\rangle +\eta_{t}\left\langle x^{*}-x_{t},\theta_{t}^{b}\right\rangle \\
 & \quad+2\eta_{t}^{2}\left(\left\Vert \theta_{t}^{u}\right\Vert _{*}^{2}-\E\left[\left\Vert \theta_{t}^{u}\right\Vert _{*}^{2}\mid\F_{t-1}\right]\right)+2\eta_{t}^{2}\E\left[\left\Vert \theta_{t}^{u}\right\Vert _{*}^{2}\mid\F_{t-1}\right]+2\eta_{t}^{2}\left\Vert \theta_{t}^{b}\right\Vert _{*}^{2}+2G^{2}\eta_{t}^{2}
\end{align*}
\end{lem}
\begin{proof}
By condition (\ref{eq:1}) and convexity,
\begin{align*}
f\left(x_{t+1}\right)-f\left(x^{*}\right) & \le\underbrace{f\left(x_{t+1}\right)-f\left(x_{t}\right)}_{\mbox{condition \eqref{eq:1}}}+\underbrace{f\left(x_{t}\right)-f\left(x^{*}\right)}_{\mbox{convexity}}\\
 & \le\left\langle \nabla f\left(x_{t}\right),x_{t+1}-x_{t}\right\rangle +\frac{L}{2}\left\Vert x_{t}-x_{t+1}\right\Vert ^{2}+G\left\Vert x_{t}-x_{t+1}\right\Vert +\left\langle \nabla f\left(x_{t}\right),x_{t}-x^{*}\right\rangle \\
 & =\left\langle \nabla f\left(x_{t}\right),x_{t+1}-x^{*}\right\rangle +\frac{L}{2}\left\Vert x_{t}-x_{t+1}\right\Vert ^{2}+G\left\Vert x_{t}-x_{t+1}\right\Vert \\
 & =\left\langle \theta_{t},x^{*}-x_{t+1}\right\rangle +\left\langle \tn(x_{t}),x_{t+1}-x^{*}\right\rangle +\frac{L}{2}\left\Vert x_{t}-x_{t+1}\right\Vert ^{2}+G\left\Vert x_{t}-x_{t+1}\right\Vert 
\end{align*}
By the optimality condition, we have
\[
\left\langle \eta_{t}\tn(x_{t})+\nabla_{x}\breg\left(x_{t+1},x_{t}\right),x^{*}-x_{t+1}\right\rangle \ge0
\]
 and thus
\[
\left\langle \eta_{t}\tn(x_{t}),x_{t+1}-x^{*}\right\rangle \leq\left\langle \nabla_{x}\breg\left(x_{t+1},x_{t}\right),x^{*}-x_{t+1}\right\rangle 
\]
Note that 
\begin{align*}
\left\langle \nabla_{x}\breg\left(x_{t+1},x_{t}\right),x^{*}-x_{t+1}\right\rangle  & =\left\langle \nabla\psi\left(x_{t+1}\right)-\nabla\psi\left(x_{t}\right),x^{*}-x_{t+1}\right\rangle \\
 & =\breg\left(x^{*},x_{t}\right)-\breg\left(x_{t+1},x_{t}\right)-\breg\left(x^{*},x_{t+1}\right)
\end{align*}
 Thus
\begin{align*}
\eta_{t}\left\langle \tn(x_{t}),x_{t+1}-x^{*}\right\rangle  & \leq\breg\left(x^{*},x_{t}\right)-\breg\left(x^{*},x_{t+1}\right)-\breg\left(x_{t+1},x_{t}\right)\\
 & \leq\breg\left(x^{*},x_{t}\right)-\breg\left(x^{*},x_{t+1}\right)-\frac{1}{2}\left\Vert x_{t+1}-x_{t}\right\Vert ^{2}
\end{align*}
 where we have used that $\breg\left(x_{t+1},x_{t}\right)\geq\frac{1}{2}\left\Vert x_{t+1}-x_{t}\right\Vert ^{2}$
by the strong convexity of $\psi$.

Combining the two inequalities, and using the assumption that $L\eta_{t}\le\frac{1}{4}$,
we obtain
\begin{align*}
 & \eta_{t}\Delta_{t+1}+\breg\left(x^{*},x_{t+1}\right)-\breg\left(x^{*},x_{t}\right)\\
 & \le\eta_{t}\left\langle \theta_{t},x^{*}-x_{t+1}\right\rangle +\frac{L\eta_{t}}{2}\left\Vert x_{t}-x_{t+1}\right\Vert ^{2}+G\eta_{t}\left\Vert x_{t}-x_{t+1}\right\Vert -\frac{1}{2}\left\Vert x_{t+1}-x_{t}\right\Vert ^{2}\\
 & \le\eta_{t}\left\langle \theta_{t},x^{*}-x_{t}\right\rangle +\eta_{t}\left\langle \theta_{t},x_{t}-x_{t+1}\right\rangle -\frac{3}{8}\left\Vert x_{t+1}-x_{t}\right\Vert ^{2}+G\eta_{t}\left\Vert x_{t}-x_{t+1}\right\Vert \\
 & \le\eta_{t}\left\langle \theta_{t},x^{*}-x_{t}\right\rangle +\eta_{t}^{2}\left\Vert \theta_{t}\right\Vert _{*}^{2}+2G^{2}\eta_{t}^{2}\\
 & \le\eta_{t}\left\langle \theta_{t}^{u}+\theta_{t}^{b},x^{*}-x_{t}\right\rangle +2\eta_{t}^{2}\left\Vert \theta_{t}^{u}\right\Vert _{*}^{2}+2\eta_{t}^{2}\left\Vert \theta_{t}^{b}\right\Vert _{*}^{2}+2G^{2}\eta_{t}^{2}
\end{align*}
We obtain the desired inequality.
\end{proof}

\section{Missing Proofs from Section \ref{sec:Clipped-Accelerated-Stochastic}}

\begin{proof}[Proof of Lemma \ref{lem:convex-acc-basic-inequality}]
We have 
\begin{align*}
f\left(y_{t+1}\right)-f\left(x^{*}\right) & =\underbrace{f\left(y_{t+1}\right)-f\left(x_{t}\right)}_{\mbox{smoothness}}+\underbrace{f\left(x_{t}\right)-f\left(x^{*}\right)}_{\mbox{convexity}}\\
 & \le\left\langle \nabla f\left(x_{t}\right),y_{t+1}-x_{t}\right\rangle +\frac{L}{2}\left\Vert y_{t+1}-x_{t}\right\Vert ^{2}\\
 & \quad+\alpha_{t}\left\langle \nabla f\left(x_{t}\right),x_{t}-x^{*}\right\rangle +\left(1-\alpha_{t}\right)\left(f\left(x_{t}\right)-f\left(x^{*}\right)\right)\\
 & =\underbrace{\left(1-\alpha_{t}\right)\left\langle \nabla f\left(x_{t}\right),y_{t}-x_{t}\right\rangle }_{\mbox{convexity}}+\alpha_{t}\left\langle \nabla f\left(x_{t}\right),z_{t+1}-x^{*}\right\rangle \\
 & \quad+\frac{L\alpha_{t}^{2}}{2}\left\Vert z_{t+1}-z_{t}\right\Vert ^{2}+\left(1-\alpha_{t}\right)\left(f\left(x_{t}\right)-f\left(x^{*}\right)\right)\\
 & \le\left(1-\alpha_{t}\right)\left(f\left(y_{t}\right)-f\left(x_{t}\right)\right)+\left(1-\alpha_{t}\right)\left(f\left(x_{t}\right)-f\left(x^{*}\right)\right)\\
 & \quad+\alpha_{t}\left\langle \theta_{t},x^{*}-z_{t+1}\right\rangle +\alpha_{t}\left\langle \tn(x_{t}),z_{t+1}-x^{*}\right\rangle +\frac{L\alpha_{t}^{2}}{2}\left\Vert z_{t+1}-z_{t}\right\Vert ^{2}\\
 & \le\left(1-\alpha_{t}\right)\left(f\left(y_{t}\right)-f\left(x^{*}\right)\right)+\alpha_{t}\left\langle \theta_{t},x^{*}-z_{t+1}\right\rangle \\
 & \quad+\alpha_{t}\left\langle \tn(x_{t}),z_{t+1}-x^{*}\right\rangle +\frac{L\alpha_{t}^{2}}{2}\left\Vert z_{t+1}-z_{t}\right\Vert ^{2}
\end{align*}
By the optimality condition, we have
\[
\left\langle \eta_{t}\tn(x_{t})+\nabla_{x}\breg\left(z_{t+1},z_{t}\right),x^{*}-z_{t+1}\right\rangle \ge0
\]
 and thus
\[
\left\langle \eta_{t}\tn(x_{t}),z_{t+1}-x^{*}\right\rangle \leq\left\langle \nabla_{x}\breg\left(z_{t+1},z_{t}\right),x^{*}-z_{t+1}\right\rangle 
\]
Note that 
\begin{align*}
\left\langle \nabla_{x}\breg\left(z_{t+1},z_{t}\right),x^{*}-z_{t+1}\right\rangle  & =\left\langle \nabla\psi\left(z_{t+1}\right)-\nabla\psi\left(z_{t}\right),x^{*}-z_{t+1}\right\rangle \\
 & =\breg\left(x^{*},z_{t}\right)-\breg\left(z_{t+1},z_{t}\right)-\breg\left(x^{*},z_{t+1}\right)
\end{align*}
 Thus
\begin{align*}
\eta_{t}\left\langle \tn(x_{t}),z_{t+1}-x^{*}\right\rangle  & \leq\breg\left(x^{*},z_{t}\right)-\breg\left(x^{*},z_{t+1}\right)-\breg\left(z_{t+1},z_{t}\right)\\
 & \leq\breg\left(x^{*},z_{t}\right)-\breg\left(x^{*},z_{t+1}\right)-\frac{1}{2}\left\Vert z_{t+1}-z_{t}\right\Vert ^{2}
\end{align*}
 where we have used that $\breg\left(z_{t+1},z_{t}\right)\geq\frac{1}{2}\left\Vert z_{t+1}-z_{t}\right\Vert ^{2}$
by the strong convexity of $\psi$.
\begin{align*}
f\left(y_{t+1}\right)-f\left(x^{*}\right) & \le\left(1-\alpha_{t}\right)\left(f\left(y_{t}\right)-f\left(x^{*}\right)\right)+\alpha_{t}\left\langle \theta_{t},x^{*}-z_{t+1}\right\rangle \\
 & +\frac{\alpha_{t}}{\eta_{t}}\breg\left(x^{*},z_{t}\right)-\frac{\alpha_{t}}{\eta_{t}}\breg\left(x^{*},z_{t+1}\right)+\left(\frac{L\alpha_{t}^{2}}{2}-\frac{\alpha_{t}}{2\eta_{t}}\right)\left\Vert z_{t+1}-z_{t}\right\Vert ^{2}
\end{align*}
Divide both sides by $\frac{\alpha_{t}}{\eta_{t}}$ and using the
condition $L\eta_{t}\alpha_{t}\le\frac{1}{2}$ we have
\begin{align*}
\frac{\eta_{t}}{\alpha_{t}}\left(f\left(y_{t+1}\right)-f\left(x^{*}\right)\right)+\breg\left(x^{*},z_{t+1}\right)-\breg\left(x^{*},z_{t}\right) & \le\frac{\eta_{t}\left(1-\alpha_{t}\right)}{\alpha_{t}}\left(f\left(y_{t}\right)-f\left(x^{*}\right)\right)+\eta_{t}\left\langle \theta_{t},x^{*}-z_{t}\right\rangle \\
 & \quad+\eta_{t}\left\langle \theta_{t},z_{t}-z_{t+1}\right\rangle -\frac{1-L\eta_{t}\alpha_{t}}{2}\left\Vert z_{t+1}-z_{t}\right\Vert ^{2}\\
 & \le\frac{\eta_{t}\left(1-\alpha_{t}\right)}{\alpha_{t}}\left(f\left(y_{t}\right)-f\left(x^{*}\right)\right)+\eta_{t}\left\langle \theta_{t},x^{*}-z_{t}\right\rangle \\
 & \quad+\frac{\eta_{t}^{2}\left\Vert \theta_{t}\right\Vert _{*}^{2}}{2\left(1-L\eta_{t}\alpha_{t}\right)}\\
 & \le\frac{\eta_{t}\left(1-\alpha_{t}\right)}{\alpha_{t}}\left(f\left(y_{t}\right)-f\left(x^{*}\right)\right)+\eta_{t}\left\langle \theta_{t}^{u}+\theta_{t}^{b},x^{*}-z_{t}\right\rangle \\
 & \quad+2\eta_{t}^{2}\left\Vert \theta_{t}^{u}\right\Vert _{*}^{2}+2\eta_{t}^{2}\left\Vert \theta_{t}^{b}\right\Vert _{*}^{2}
\end{align*}
as needed.
\end{proof}

\begin{thm}
Assume that $f$ satisfies Assumption (1), (2), (3), (4) and (5').
Let $\gamma=\max\left\{ \log\frac{1}{\delta};1\right\} $; and $R_{1}=\sqrt{2\breg\left(x^{*},x_{1}\right)}$.
For unknown $T$, we choose $c_{t}$, $\lambda_{t}$ and $\eta_{t}$
such that
\begin{align*}
c_{t} & =\max\left\{ 10^{4};\frac{4\left(t+1\right)\left(\frac{52t\left(1+\log t\right)^{2}}{\gamma}\right)^{\frac{1}{p}}\sigma}{\gamma LR_{1}}\right\} \\
\lambda_{t} & =\frac{c_{t}R_{1}\gamma L\alpha_{t}}{8}=\max\left\{ \frac{10^{4}R_{1}\gamma L}{4(t+1)};\left(\frac{52t\left(1+\log t\right)^{2}}{\gamma}\right)^{1/p}\sigma\right\} \\
\eta_{t} & =\frac{1}{3c_{t}\gamma^{2}L\alpha_{t}}=\frac{R_{1}}{24\gamma}\min\left\{ \frac{4(t+1)}{10^{4}R_{1}\gamma L};\left(\frac{52t\left(1+\log t\right)^{2}}{\gamma}\right)^{-1/p}\sigma^{-1}\right\} 
\end{align*}
Then with probability at least $1-\delta$
\begin{align*}
f\left(y_{T+1}\right)-f\left(x^{*}\right) & \le6\max\left\{ 10^{4}L\gamma^{2}R_{1}^{2}(T+1)^{-2};4R_{1}\left(T+1\right)^{-1}\left(52t\left(1+\log t\right)^{2}\right)^{\frac{1}{p}}\gamma^{\frac{p-1}{p}}\sigma\right\} .
\end{align*}
\end{thm}
\begin{proof}
Following the similar steps to the proof of Theorem \ref{thm:convex-accelerated-main-convergence},
and noticing that $\left(c_{t}\right)$ is a increasing sequence,
we obtain the convergence rate.
\end{proof}

\section{Missing Proofs from Section \ref{sec:Clipped-Stochastic-Gradient}}

\begin{proof}[Proof of Lemma \ref{lem:nonconvex-basic-analysis}]
By the smoothness of $f$ and the update $x_{t+1}=x_{t}-\frac{1}{\eta_{t}}\tn(x_{t})$
we have

\begin{align*}
 & f(x_{t+1})-f(x_{t})\\
\le & \left\langle \nabla f(x_{t}),x_{t+1}-x_{t}\right\rangle +\frac{L}{2}\left\Vert x_{t+1}-x_{t}\right\Vert ^{2}\\
= & -\eta_{t}\left\langle \nabla f(x_{t}),\tn(x_{t})\right\rangle +\frac{L\eta_{t}^{2}}{2}\left\Vert \tn(x_{t})\right\Vert ^{2}\\
= & -\eta_{t}\left\langle \nabla f(x_{t}),\theta_{t}+\nabla f(x_{t})\right\rangle +\frac{L\eta_{t}^{2}}{2}\left\Vert \theta_{t}+\nabla f(x_{t})\right\Vert ^{2}\\
= & -\eta_{t}\left\Vert \nabla f(x_{t})\right\Vert ^{2}-\eta_{t}\left\langle \nabla f(x_{t}),\theta_{t}\right\rangle +\frac{L\eta_{t}^{2}}{2}\left\Vert \theta_{t}\right\Vert ^{2}+\frac{L\eta_{t}^{2}}{2}\left\Vert \nabla f(x_{t})\right\Vert ^{2}+L\eta_{t}^{2}\left\langle \nabla f(x_{t}),\theta_{t}\right\rangle \\
= & -\left(\eta_{t}-\frac{L\eta_{t}^{2}}{2}\right)\left\Vert \nabla f(x_{t})\right\Vert ^{2}+\frac{L\eta_{t}^{2}}{2}\left\Vert \theta_{t}\right\Vert ^{2}+\left(L\eta_{t}^{2}-\eta_{t}\right)\left\langle \nabla f(x_{t}),\theta_{t}\right\rangle \\
= & -\left(\eta_{t}-\frac{L\eta_{t}^{2}}{2}\right)\left\Vert \nabla f(x_{t})\right\Vert ^{2}+\frac{L\eta_{t}^{2}}{2}\left\Vert \theta_{t}\right\Vert ^{2}+\underbrace{\left(L\eta_{t}^{2}-\eta_{t}\right)}_{\le0}\left\langle \nabla f(x_{t}),\theta_{t}^{u}+\theta_{t}^{b}\right\rangle .
\end{align*}
Using Cauchy-Schwarz, we have $\left\langle \nabla f(x_{t}),\theta_{t}^{b}\right\rangle \le\frac{1}{2}\left\Vert \nabla f(x_{t})\right\Vert ^{2}+\frac{1}{2}\left\Vert \theta_{t}^{b}\right\Vert ^{2}$
thus, we derive,
\begin{align*}
\Delta_{t+1}-\Delta_{t} & \le-\left(\frac{2\eta_{t}-L\eta_{t}^{2}}{2}\right)\left\Vert \nabla f(x_{t})\right\Vert ^{2}+\frac{L\eta_{t}^{2}}{2}\left\Vert \theta_{t}\right\Vert ^{2}+\left(L\eta_{t}^{2}-\eta_{t}\right)\left\langle \nabla f(x_{t}),\theta_{t}^{u}\right\rangle \\
 & \quad+\frac{\eta_{t}-L\eta_{t}^{2}}{2}\left\Vert \nabla f(x_{t})\right\Vert ^{2}+\frac{\eta_{t}-L\eta_{t}^{2}}{2}\left\Vert \theta_{t}^{b}\right\Vert ^{2}\\
 & \le-\frac{\eta_{t}}{2}\left\Vert \nabla f(x_{t})\right\Vert ^{2}+\frac{L\eta_{t}^{2}}{2}\left\Vert \theta_{t}\right\Vert ^{2}+\left(L\eta_{t}^{2}-\eta_{t}\right)\left\langle \nabla f(x_{t}),\theta_{t}^{u}\right\rangle +\frac{\eta_{t}}{2}\left\Vert \theta_{t}^{b}\right\Vert ^{2}\\
 & \le-\frac{\eta_{t}}{2}\left\Vert \nabla f(x_{t})\right\Vert ^{2}+L\eta_{t}^{2}\left\Vert \theta_{t}^{u}\right\Vert ^{2}+\left(L\eta_{t}^{2}-\eta_{t}\right)\left\langle \nabla f(x_{t}),\theta_{t}^{u}\right\rangle +\left(L\eta_{t}^{2}+\frac{\eta_{t}}{2}\right)\left\Vert \theta_{t}^{b}\right\Vert ^{2}\\
 & \le-\frac{\eta_{t}}{2}\left\Vert \nabla f(x_{t})\right\Vert ^{2}+L\eta_{t}^{2}\left\Vert \theta_{t}^{u}\right\Vert ^{2}+\left(L\eta_{t}^{2}-\eta_{t}\right)\left\langle \nabla f(x_{t}),\theta_{t}^{u}\right\rangle +\frac{3\eta_{t}}{2}\left\Vert \theta_{t}^{b}\right\Vert ^{2}
\end{align*}
where the last inequality is due to $\eta_{t}\le\frac{1}{L}$. Rearranging,
we obtain the lemma.
\end{proof}

\begin{proof}[Proof of Lemma \ref{lem:nonconvex-key-inequality}]
We have

\begin{align*}
 & \E\left[\exp\left(Z_{t}\right)\mid\F_{t-1}\right]\exp\left(\left(3z_{t}^{2}L\eta_{t}^{2}\Delta_{t}+6L^{2}z_{t}^{2}\eta_{t}^{4}\lambda_{t}^{2}\right)\E\left[\left\Vert \theta_{t}^{u}\right\Vert ^{2}\mid\F_{t-1}\right]\right)\\
 & \overset{(a)}{\le}\E\left[\exp\left(z_{t}\left(\left(L\eta_{t}^{2}-\eta_{t}\right)\left\langle \nabla f(x_{t}),\theta_{t}^{u}\right\rangle +L\eta_{t}^{2}\left(\left\Vert \theta_{t}^{u}\right\Vert ^{2}-\E\left[\left\Vert \theta^{u}\right\Vert ^{2}\mid\F_{t-1}\right]\right)\right)\right)\mid\F_{t-1}\right]\\
 & \overset{(b)}{\le}\exp\left(\E\left[\frac{3}{4}\left(z_{t}\left(\left(L\eta_{t}^{2}-\eta_{t}\right)\left\langle \nabla f(x_{t}),\theta_{t}^{u}\right\rangle +L\eta_{t}^{2}\left(\left\Vert \theta_{t}^{u}\right\Vert ^{2}-\E\left[\left\Vert \theta^{u}\right\Vert ^{2}\mid\F_{t-1}\right]\right)\right)\right)^{2}\mid\F_{t-1}\right]\right)\\
 & \overset{(c)}{\le}\exp\left(\E\left[\frac{3}{2}z_{t}^{2}\eta_{t}^{2}\left\Vert \nabla f(x_{t})\right\Vert ^{2}\left\Vert \theta_{t}^{u}\right\Vert ^{2}\mid\F_{t-1}\right]+\E\left[\frac{3}{2}L^{2}z_{t}^{2}\eta_{t}^{4}\left\Vert \theta_{t}^{u}\right\Vert ^{4}\mid\F_{t-1}\right]\right)\\
 & \overset{(d)}{\le}\exp\left(3z_{t}^{2}L\eta_{t}^{2}\Delta_{t}\E\left[\left\Vert \theta_{t}^{u}\right\Vert ^{2}\mid\F_{t-1}\right]+6L^{2}z_{t}^{2}\eta_{t}^{4}\lambda_{t}^{2}\E\left[\left\Vert \theta_{t}^{u}\right\Vert ^{2}\mid\F_{t-1}\right]\right)\\
 & =\exp\left(\left(3z_{t}^{2}L\eta_{t}^{2}\Delta_{t}+6L^{2}z_{t}^{2}\eta_{t}^{4}\lambda_{t}^{2}\right)\E\left[\left\Vert \theta_{t}^{u}\right\Vert ^{2}\mid\F_{t-1}\right]\right)
\end{align*}
For $(a)$ we use Lemma \ref{lem:nonconvex-basic-analysis}. For $(b)$
we use Lemma \ref{lem:moment-inequality}. Notice that 
\begin{align*}
\E\left[\left\langle \nabla f(x_{t}),\theta_{t}^{u}\right\rangle \right] & =\E\left[\left\Vert \theta_{t}^{u}\right\Vert _{*}^{2}-\E\left[\left\Vert \theta_{t}^{u}\right\Vert _{*}^{2}\mid\F_{t-1}\right]\right]=0
\end{align*}
and since $\left\Vert \theta_{t}^{u}\right\Vert \le2\lambda_{t}$
and $\left\Vert \nabla f(x_{t})\right\Vert \le\sqrt{2L\Delta_{t}}$
for an $L$-smooth function, we have
\begin{align*}
 & \left|\left(L\eta_{t}^{2}-\eta_{t}\right)\left\langle \nabla f(x_{t}),\theta_{t}^{u}\right\rangle +L\eta_{t}^{2}\left(\left\Vert \theta_{t}^{u}\right\Vert ^{2}-\E\left[\left\Vert \theta^{u}\right\Vert ^{2}\mid\F_{t-1}\right]\right)\right|\\
\le & 2\eta_{t}\lambda_{t}\left\Vert \nabla f(x_{t})\right\Vert +L\eta_{t}^{2}\left(\left\Vert \theta_{t}^{u}\right\Vert ^{2}+\E\left[\left\Vert \theta^{u}\right\Vert ^{2}\mid\F_{t-1}\right]\right)\\
\le & 2\eta_{t}\lambda_{t}\left\Vert \nabla f(x_{t})\right\Vert +8L\eta_{t}^{2}\lambda_{t}^{2}\\
\le & 2\eta_{t}\lambda_{t}\sqrt{2L\Delta_{t}}+8L\eta_{t}^{2}\lambda_{t}^{2}
\end{align*}
Thus $z_{t}\le\frac{1}{2\eta_{t}\lambda_{t}\sqrt{2L\Delta_{t}}+8L\eta_{t}^{2}\lambda_{t}^{2}}$.
For $(c)$ we use $(a+b)^{2}\le2a^{2}+2b^{2}$ and $\E\left[\left(X-\E\left[X\right]\right)^{2}\right]\le\E\left[X^{2}\right]$.
For $(d)$ we use $\left\Vert \nabla f(x_{t})\right\Vert ^{2}\le2L\Delta_{t}$
and $\left\Vert \theta_{t}^{u}\right\Vert \le2\lambda_{t}$. We obtain
\begin{align*}
\E\left[\exp\left(Z_{t}\right)\mid\F_{t-1}\right] & \le1
\end{align*}
Therefore 
\begin{align*}
\E\left[\exp\left(S_{t}\right)\mid\F_{t-1}\right] & =\exp\left(S_{t-1}\right)\E\left[\exp\left(Z_{t}\right)\mid\F_{t-1}\right]\\
 & \le\exp\left(S_{t-1}\right)
\end{align*}
 which means $(S_{t})_{t\ge1}$ is a supermartingale. By Ville's inequality,
we have, for all $k\ge1$ 
\begin{align*}
\Pr\left[S_{k}\ge\log\frac{1}{\delta}\right] & \le\delta\E\left[\exp\left(S_{1}\right)\right]\le\delta
\end{align*}
In other words, with probability at least $1-\delta$, for all $k\ge1$
\begin{align*}
\sum_{t=1}^{k}Z_{t} & \le\log\frac{1}{\delta}
\end{align*}
Plugging in the definition of $Z_{t}$ we have
\begin{align*}
 & \frac{1}{2}\sum_{t=1}^{k}z_{t}\eta_{t}\left\Vert \nabla f(x_{t})\right\Vert ^{2}+\sum_{t=1}^{k}\left(z_{t}\Delta_{t+1}-z_{t}\Delta_{t}\right)\\
\le & \log\frac{1}{\delta}+\sum_{t=1}^{k}\frac{3z_{t}\eta_{t}}{2}\left\Vert \theta_{t}^{b}\right\Vert ^{2}\\
 & +\sum_{t=1}^{k}\left(\left(3z_{t}^{2}L\eta_{t}^{2}\Delta_{t}+6L^{2}z_{t}^{2}\eta_{t}^{4}\lambda_{t}^{2}+z_{t}L\eta_{t}^{2}\right)\E\left[\left\Vert \theta_{t}^{u}\right\Vert ^{2}\mid\F_{t-1}\right]\right)
\end{align*}
Note that we have $z_{t}$ is a decreasing sequence, hence the $\mbox{LHS}$
of the above inequality can be bounded by
\begin{align*}
\mbox{LHS} & =\frac{1}{2}\sum_{t=1}^{k}z_{t}\eta_{t}\left\Vert \nabla f(x_{t})\right\Vert ^{2}+z_{k}\Delta_{k+1}-z_{1}\Delta_{1}+\sum_{t=2}^{k}\left(z_{k-1}-z_{k}\right)\Delta_{k}\\
 & \ge\frac{1}{2}\sum_{t=1}^{k}z_{t}\eta_{t}\left\Vert \nabla f(x_{t})\right\Vert ^{2}+z_{k}\Delta_{k+1}-z_{1}\Delta_{1}
\end{align*}
We obtain the desired inequality.
\end{proof}

\begin{proof}[Proof of Proposition \ref{prop:nonconvex-general-choice}]
We will prove by induction on $k$ that
\begin{align*}
\frac{1}{2}\sum_{i=1}^{k}\eta_{i}\left\Vert \nabla f(x_{i})\right\Vert ^{2}+\Delta_{k+1} & \le\left(\sqrt{\Delta_{1}}+2\sqrt{A}C_{1}\right)^{2}
\end{align*}

The base case $k=0$ is trivial. Suppose the statement is true for
all $t\le k\le\ell$. Now we show for $k+1$. Recall that 
\begin{align*}
z_{t} & =\frac{1}{2P_{t}\eta_{t}\lambda_{t}\max_{i\le t}\sqrt{2L\Delta_{i}}+8Q_{t}L\eta_{t}^{2}\lambda_{t}^{2}}
\end{align*}
Let us choose 
\begin{align*}
P_{t} & =\frac{C_{1}}{\lambda_{t}\eta_{t}\sqrt{2L}}\ge1\\
Q_{t} & =\frac{C_{1}^{2}\sqrt{A}}{2L\eta_{t}^{2}\lambda_{t}^{2}}\ge1
\end{align*}
we have
\begin{align*}
z_{t} & =\frac{1}{2C_{1}\max_{i\le t}\sqrt{\Delta_{i}}+4C_{1}^{2}\sqrt{A}}
\end{align*}
Now we can notice that $(z_{t})_{t\ge1}$ is a decreasing sequence.
By the induction hypothesis
\begin{align*}
\frac{z_{t}}{z_{k}} & \le\frac{2C_{1}\left(\sqrt{\Delta_{1}}+2\sqrt{A}C_{1}\right)+4C_{1}^{2}\sqrt{A}}{2C_{1}\sqrt{\Delta_{1}}+4C_{1}^{2}\sqrt{A}}\\
 & =\frac{\sqrt{\Delta_{1}}+4\sqrt{A}C_{1}}{\sqrt{\Delta_{1}}+2\sqrt{A}C_{1}}\le2
\end{align*}
By the choice of $\lambda_{t}$, for all $t\le k$, $\left\Vert \n(x_{t})\right\Vert \le\frac{\lambda_{t}}{2}$,
we can apply Lemma \ref{lem:bias-bounds} and have 
\begin{align*}
\left\Vert \theta_{t}^{b}\right\Vert  & \le4\sigma^{p}\lambda_{t}^{1-p};\\
\E\left[\left\Vert \theta_{t}^{u}\right\Vert ^{2}\mid\F_{t-1}\right] & \le40\sigma^{p}\lambda_{t}^{2-p}.
\end{align*}
Thus 
\begin{align*}
 & \frac{1}{2}z_{k}\sum_{t=1}^{k}\eta_{t}\left\Vert \nabla f(x_{t})\right\Vert ^{2}+z_{k}\Delta_{k+1}\\
\le & z_{1}\Delta_{1}+\log\frac{1}{\delta}+\sum_{t=1}^{k}\frac{3z_{t}\eta_{t}}{2}\left\Vert \theta_{t}^{b}\right\Vert ^{2}\\
 & +\sum_{t=1}^{k}\left(\left(3z_{t}^{2}L\eta_{t}^{2}\Delta_{t}+6L^{2}z_{t}^{2}\eta_{t}^{4}\lambda_{t}^{2}+z_{t}L\eta_{t}^{2}\right)\E\left[\left\Vert \theta_{t}^{u}\right\Vert ^{2}\mid\F_{t-1}\right]\right)\\
\le & z_{1}\Delta_{1}+\log\frac{1}{\delta}+24\sigma^{2p}\sum_{t=1}^{k}z_{t}\eta_{t}\lambda_{t}^{2}\left(\frac{1}{\lambda_{t}}\right)^{2p}\\
 & +40\sigma^{p}\sum_{t=1}^{k}\left(\left(3z_{t}^{2}\Delta_{t}+6z_{t}^{2}L\eta_{t}^{2}\lambda_{t}^{2}+z_{t}\right)L\eta_{t}^{2}\lambda_{t}^{2}\left(\frac{1}{\lambda_{t}}\right)^{p}\right)
\end{align*}
Since $\frac{z_{t}}{z_{k}}\le2$ we have
\begin{align*}
 & \frac{1}{2}\sum_{t=1}^{k}\eta_{t}\left\Vert \nabla f(x_{t})\right\Vert ^{2}+\Delta_{k+1}\\
\le & \frac{z_{1}\Delta_{1}}{z_{k}}+\frac{1}{z_{k}}\log\frac{1}{\delta}+48\sigma^{2p}\sum_{t=1}^{k}\eta_{t}\lambda_{t}^{2}\left(\frac{1}{\lambda_{t}}\right)^{2p}\\
 & +80\sigma^{p}\sum_{t=1}^{k}\left(\left(3z_{t}\Delta_{t}+6z_{t}L\eta_{t}^{2}\lambda_{t}^{2}+1\right)L\eta_{t}^{2}\lambda_{t}^{2}\left(\frac{1}{\lambda_{t}}\right)^{p}\right)\\
\overset{(a)}{\le} & \frac{\sqrt{\Delta_{1}}+4\sqrt{A}C_{1}}{\sqrt{\Delta_{1}}+2\sqrt{A}C_{1}}\Delta_{1}+2C_{1}\left(\sqrt{\Delta_{1}}+4\sqrt{A}C_{1}\right)\log\frac{1}{\delta}+48\sigma^{2p}C_{2}\sum_{t=1}^{k}L\eta_{t}^{2}\lambda_{t}^{2}\left(\frac{1}{\lambda_{t}}\right)^{p}\\
 & +80\sigma^{p}\sum_{t=1}^{k}\left(\left(\frac{3\left(\sqrt{\Delta_{1}}+2\sqrt{A}C_{1}\right)^{2}}{2C_{1}\left(\sqrt{\Delta_{1}}+2\sqrt{A}C_{1}\right)}+\frac{6}{8Q_{t}}+1\right)L\eta_{t}^{2}\lambda_{t}^{2}\left(\frac{1}{\lambda_{t}}\right)^{p}\right)\\
\overset{(b)}{\le} & \Delta_{1}+2\sqrt{\Delta_{1}}\sqrt{A}C_{1}+2C_{1}\left(\sqrt{\Delta_{1}}+4\sqrt{A}C_{1}\right)\log\frac{1}{\delta}+48\sigma^{2p}C_{2}C_{3}\\
 & +80\sigma^{p}\left(\frac{3\left(\sqrt{\Delta_{1}}+2\sqrt{A}C_{1}\right)}{2C_{1}}+\frac{7}{4}\right)C_{3}\\
\le & \Delta_{1}+2\sqrt{\Delta_{1}}\sqrt{A}C_{1}+2C_{1}\left(\sqrt{\Delta_{1}}+4\sqrt{A}C_{1}\right)\left(\log\frac{1}{\delta}+\frac{60\sigma^{p}C_{3}}{C_{1}^{2}}\right)\\
 & +48\sigma^{2p}C_{2}C_{3}+140\sigma^{p}C_{3}\\
\overset{(c)}{\le} & \Delta_{1}+2\sqrt{\Delta_{1}}\sqrt{A}C_{1}+2C_{1}\left(\sqrt{\Delta_{1}}+4\sqrt{A}C_{1}\right)\frac{\sqrt{A}}{8}+AC_{1}^{2}\\
\le & \left(\sqrt{\Delta_{1}}+2\sqrt{A}C_{1}\right)^{2}
\end{align*}
For $(a)$ we use $\left(\frac{1}{\lambda_{t}}\right)^{p}\le C_{2}L\eta_{t}$
and the induction hypothesis. For $(b)$ we use $\sum_{t=1}^{T}L\left(\frac{1}{\lambda_{t}}\right)^{p}\lambda_{t}^{2}\eta_{t}^{2}\le C_{3}$
and $Q_{t}\ge1$. For $(c)$ we have
\begin{align*}
\log\frac{1}{\delta}+\frac{60\sigma^{p}C_{3}}{C_{1}^{2}} & \le\frac{\sqrt{A}}{8}\\
48\sigma^{2p}C_{2}C_{3}+140\sigma^{p}C_{3} & \le AC_{1}^{2}
\end{align*}
since
\begin{align*}
A & \ge64\left(\log\frac{1}{\delta}+\frac{60\sigma^{p}C_{3}}{C_{1}^{2}}\right)^{2}+\frac{48\sigma^{2p}C_{2}C_{3}+140\sigma^{p}C_{3}}{C_{1}^{2}}.
\end{align*}
This concludes the proof.
\end{proof}

\begin{lem}
\label{lem:nonconvex-condition-verification}The choices of $\eta_{t}$
and $\lambda_{t}$ in Theorem \ref{thm:nonconvex-convergence} satisfy
the condition (1)-(3) of Proposition \ref{prop:nonconvex-general-choice}
for 
\begin{align*}
C_{1} & =\frac{\sqrt{\Delta_{1}}}{4\sqrt{2}\gamma}\\
C_{2} & =\frac{1}{\sigma^{p}}\\
C_{3} & =\frac{\Delta_{1}}{2048\sigma^{p}\gamma}.
\end{align*}
\end{lem}
\begin{proof}
We verify for the first case. The second follows exactly the same.
First, we have $p>1$ hence
\begin{align*}
\eta_{t}\lambda_{t}\sqrt{2L} & =\frac{\sqrt{\Delta_{1}}T^{\frac{1-p}{3p-2}}}{8\sqrt{L}\gamma}\sqrt{2L}\le\frac{\sqrt{\Delta_{1}}}{4\sqrt{2}\gamma}=C_{1}
\end{align*}

Since $\eta_{t}=\frac{\sqrt{\Delta_{1}}T^{\frac{1-p}{3p-2}}}{8\lambda_{t}\sqrt{L}\gamma}$,
$p>1$ and $\lambda_{t}\ge\left(\frac{8\gamma}{\sqrt{L\Delta_{1}}}\right)^{\frac{1}{p-1}}T^{\frac{1}{3p-2}}\sigma^{\frac{p}{p-1}}$
\begin{align*}
\eta_{t}\lambda_{t}^{p} & =\frac{\sqrt{\Delta_{1}}T^{\frac{1-p}{3p-2}}}{8\sqrt{L}\gamma}\lambda_{t}^{p-1}\\
 & \ge\frac{\sqrt{\Delta_{1}}T^{\frac{1-p}{3p-2}}}{8\sqrt{L}\gamma}\frac{8\gamma}{\sqrt{L\Delta_{1}}}T^{\frac{p-1}{3p-2}}\sigma^{p}\\
 & =\frac{\sigma^{p}}{L}
\end{align*}
which gives 
\begin{align*}
\frac{1}{L\eta_{t}}\left(\frac{1}{\lambda_{t}}\right)^{p} & \le\frac{1}{\sigma^{p}}=C_{2}.
\end{align*}

Finally, we have $\lambda_{t}\ge32^{1/p}\sigma T^{\frac{1}{3p-2}}$
hence 
\begin{align*}
\left(\frac{1}{\lambda_{t}}\right)^{p}T^{\frac{p}{3p-2}} & \le\frac{1}{32\sigma^{p}}.
\end{align*}

Therefore, 
\begin{align*}
\sum_{t=1}^{T}L\left(\frac{1}{\lambda_{t}}\right)^{p}\lambda_{t}^{2}\eta_{t}^{2} & =\sum_{t=1}^{T}L\left(\frac{1}{\lambda_{t}}\right)^{p}\left(\frac{\sqrt{\Delta_{1}}T^{\frac{1-p}{3p-2}}}{8\sqrt{L}\gamma}\right)^{2}\\
 & \le TL\left(\frac{1}{\lambda_{t}}\right)^{p}T^{\frac{2-2p}{3p-2}}\frac{\Delta_{1}}{64L\gamma}\\
 & =\left(\frac{1}{\lambda_{t}}\right)^{p}T^{\frac{p}{3p-2}}\frac{\Delta_{1}}{64\gamma^{2}}\\
 & \leq\frac{1}{32\sigma^{p}}\frac{\Delta_{1}}{64\gamma^{2}}\leq\frac{\Delta_{1}}{2048\sigma^{p}\gamma}.
\end{align*}

\end{proof}

\begin{thm}
Assume that $f$ satisfies Assumption (1'), (2), (3), (4). Let $\gamma=\max\left\{ \log\frac{1}{\delta};1\right\} $
and $\Delta_{1}=f(x_{1})-f^{*}$. For unknown $T$, we choose $\lambda_{t}$
and $\eta_{t}$ such that 
\begin{align*}
\lambda_{t} & =\max\left\{ \left(\frac{8\gamma}{\sqrt{L\Delta_{1}}}\right)^{\frac{1}{p-1}}\left(2t\left(1+\log t\right)^{2}\right)^{\frac{1}{3p-2}}\sigma^{\frac{p}{p-1}};2\sqrt{90L\Delta_{1}};32^{\frac{1}{p}}\sigma\left(2t\left(1+\log t\right)^{2}\right)^{\frac{1}{3p-2}}\right\} \\
\eta_{t} & =\frac{\sqrt{\Delta_{1}}\left(2t\left(1+\log t\right)^{2}\right)^{\frac{1-p}{3p-2}}}{8\lambda_{t}\sqrt{L}\gamma}
\end{align*}
Then with probability at least $1-\delta$
\begin{align*}
\frac{1}{T}\sum_{t=1}^{T}\left\Vert \n(x_{t})\right\Vert ^{2} & \le720\sqrt{\Delta_{1}L}\gamma\max\Bigg\{\left(\frac{8\gamma}{\sqrt{L\Delta_{1}}}\right)^{\frac{1}{p-1}}\left(2\left(1+\log T\right)^{2}\right)^{\frac{p}{3p-2}}\sigma^{\frac{p}{p-1}}T^{\frac{2-2p}{3p-2}};\\
 & \qquad2\sqrt{90L\Delta_{1}}\left(2\left(1+\log T\right)^{2}\right)^{\frac{p-1}{3p-2}}T^{\frac{1-2p}{3p-2}};32^{\frac{1}{p}}\sigma\left(2\left(1+\log T\right)^{2}\right)^{\frac{p}{3p-2}}T^{\frac{2-2p}{3p-2}}\Bigg\}.
\end{align*}
\end{thm}
\begin{proof}
Following exactly the same steps as in the case with known $T$ and
noticing that $\eta_{t}$ is decreasing, we obtain the convergence
guarantee.
\end{proof}

\end{document}